\documentclass[12pt]{article}

\usepackage{latexsym,amsfonts,amsmath,amssymb,mathrsfs}

\newtheorem{theorem}{Theorem}
\newtheorem{lemma}{Lemma}

\newtheorem{remark}{Remark}
\newtheorem{proposition}{Proposition}
\newtheorem{definition}{Definition}
\newtheorem{corollary}{Corollary}

\newcommand{\rd}{\,\mathrm{d}}

\newcommand{\bsx}{\boldsymbol{x}}
\newcommand{\bsy}{\boldsymbol{y}}

\newcommand{\bsl}{\boldsymbol{l}}
\newcommand{\bsk}{\boldsymbol{k}}
\newcommand{\bst}{\boldsymbol{t}}

\newcommand{\bssigma}{\boldsymbol{\sigma}}

\newcommand{\bszero}{\boldsymbol{0}}

\newcommand{\nat}{\mathbb{N}}

\newcommand{\LL}{{\cal L}}

\newcommand{\FF}{\mathbb{F}}
\newcommand{\NN}{\mathbb{N}}
\newcommand{\QQ}{\mathbb{Q}}

\newcommand{\Dcal}{\mathcal{D}}
\newcommand{\Jcal}{\mathcal{J}}

\newcommand{\Mod}{\operatorname{mod}}
\newcommand{\mypmod}[1]{\,(\Mod\,#1)}
\newcommand{\wal}{{\rm wal}}
\newcommand{\cP}{\mathcal{P}}
\newcommand{\cS}{\mathcal{S}}
\newenvironment{proof}{\begin{trivlist}
    \item[\hskip\labelsep{\it Proof.}]}{$\hfill\Box$\end{trivlist}}
\newcommand{\satop}[2]{\stackrel{\scriptstyle{#1}}{\scriptstyle{#2}}}

\allowdisplaybreaks

\begin{document}

\title{Optimal $\mathcal{L}_2$ discrepancy bounds for higher order digital sequences over the finite field $\mathbb{F}_2$}

\author{Josef Dick \thanks{The first author is supported by a Queen Elizabeth 2 Fellowship of the Australian Research Council.} and Friedrich Pillichshammer \thanks{The second author is supported by the Austrian Research Foundation (FWF), Project S9609.}}

\date{}

\maketitle

\begin{abstract}
We show that the $\LL_2$ discrepancy of the explicitly constructed infinite sequences of points $(\bsx_0,\bsx_1, \bsx_2, \ldots)$ over $\mathbb{F}_2$ introduced in [J. Dick, Walsh spaces containing smooth functions and quasi-Monte Carlo rules of arbitrary high order. SIAM J. Numer. Anal., {\bf 46}, 1519--1553, 2008] satisfies $$\mathcal{L}_{2,N}(\{\boldsymbol{x}_0,\boldsymbol{x}_1,\ldots, \boldsymbol{x}_{N-1}\}) \le C_s N^{-1} (\log N)^{s/2} \quad \mbox{for all } N \ge 2,$$ and $$\mathcal{L}_{2,2^m}(\{\boldsymbol{x}_0,\boldsymbol{x}_1,\ldots, \boldsymbol{x}_{2^m-1}\}) \le C_s 2^{-m} m^{(s-1)/2} \quad \mbox{for all } m \ge 1,$$ where $C_s > 0$ is a constant independent of $N$ and $m$. These results are best possible by lower bounds in [P.D. Proinov, On the $L^2$ discrepancy of some infinite sequences. Serdica, {\bf 11}, 3--12, 1985] and [K. F. Roth, On irregularities of distribution. Mathematika, {\bf 1}, 73--79, 1954]. Further, for every $N \ge 2$ we explicitly construct finite point sets $\{\boldsymbol{y}_0,
\ldots, \boldsymbol{y}_{N-1}\}$ in $[0,1)^s$ such that $$\mathcal{L}_{2,N}(\{\boldsymbol{y}_0,\boldsymbol{y}_1,\ldots, \boldsymbol{y}_{N-1}\}) \le C_s N^{-1} (\log N)^{(s-1)/2}.$$   Another solution for finite point sets by a different construction was previously shown in [W. W. L. Chen and M. M. Skriganov, Explicit constructions in the classical mean squares problem in irregularity of point distribution.  J. Reine Angew. Math., {\bf 545}, 67--95, 2002].
\end{abstract}



\section{Introduction and statement of the main results}

We study equidistribution properties of point sets in the
$s$-dimensional unit-cube $[0,1)^s$ measured by their $\LL_2$ discrepancy (see \cite{BC, DP10, DT97,
kuinie,mat}). For a finite set $\cP_{N,s} =\{\bsx_0,\ldots
,\bsx_{N-1}\}$ of points in the $s$-dimensional unit-cube $[0,1)^s$
the local discrepancy function is defined as
$$\Delta(t_1,\ldots,t_s)=\frac{A_N([\bszero,\bst),\cP_{N,s})}{N}-t_1\cdots t_s,$$ where $\bst=(t_1,\ldots,t_s) \in [0,1]^s$ and
$A_N([\bszero,\bst), \cP_{N,s})$ denotes the number of indices $n$ with
$\bsx_n \in [0,t_1)\times \dots \times [0,t_s) =: [\bszero, \bst)$. The discrepancy function measures the difference of the portion of points in an axis parallel box containing the origin and the volume of this box. Hence it is a measure of the irregularity of distribution of a point set in $[0,1)^s$.

The {\it $\LL_2$ discrepancy} of $\cP_{N,s}$ is defined as
\begin{eqnarray}\label{L2-definition}
\LL_{2,N}(P_{N,s})=\left(\int_{[0,1]^{s}} |\Delta(\bst)|^2 \rd\bst\right)^{1/2}.
\end{eqnarray}
For an infinite sequence $\cS_s= (\bsx_0, \bsx_1, \ldots)$ in $[0,1)^s$ the $\LL_2$ discrepancy $\LL_{2,N}(\cS_s)$ is the $\LL_2$ discrepancy of the first $N$ elements of $\cS_s$.

It is well known that a sequence is uniformly distributed modulo one if and only if its $\LL_2$ discrepancy tends to zero for growing $N$. Furthermore, the $\LL_2$ discrepancy can also be linked to the integration error of a quasi-Monte Carlo rule, see, e.g. \cite{DP10,Nie73,SloWo} for the error in the worst-case setting and \cite{Woz} for the average case setting.

A lower bound on the $\LL_2$ discrepancy of {\it finite} point sets has been shown by Roth~\cite{Roth} which states that for any $s \in \NN$ (the set of positive integers) there exists a number $c_s> 0$ depending only on $s$, such that for every point set $\cP_{N,s}$ in $[0,1)^s$ consisting of $N \ge 2$ points we have
\begin{equation}\label{lbound_roth}
\mathcal{L}_2(\cP_{N,s}) \ge c_s \frac{(\log N)^{(s-1)/2}}{N}.
\end{equation}
This lower bound is best possible in the order of magnitude in $N$
as shown first by Davenport~\cite{dav} for $s=2$ and then by
Roth~\cite{roth2,Roth4} for arbitrary dimensions $s \in \NN$. Other constructions of point sets with optimal $\LL_2$ discrepancy were found by Chen~\cite{C80,C83}, Dobrovolski\v{i}~\cite{Do}, Frolov~\cite{Frolov} and Skriganov~\cite{S89, S94}. Davenport used point sets consisting of the $2N$ elements $(\{\pm n \alpha\},n/N)$ for $1 \le n \le N$, where $N \in \NN$ and $\alpha$ has a continued fraction expansion with
bounded partial quotients. Further examples of two-dimensional point
sets with best possible order of $\LL_2$ discrepancy can be found in
\cite{FauPi09a,FauPi09,FauPiPriSch09,KriPi2006,lp,pro1988a}. On
the other hand, Roth's~\cite{Roth4} proof for dimensions $s \ge 2$ is a pure existence result obtained by averaging arguments as are the constructions in \cite{C80, C83, Do, Frolov, S89, S94}. Explicit constructions of point sets achieving the best possible order of convergence have been a longstanding open problem. Finally, a solution was given by Chen and Skriganov~\cite{CS02} who, for every integer $N \ge 2$ and every dimension $s \in \NN$, gave for the first time explicit constructions of finite point sets consisting of $N$ points in $[0,1)^s$ whose $\LL_2$ discrepancy achieves an order of convergence of $(\log N)^{(s-1)/2}/N$. Their construction uses a finite field $\mathbb{F}_p$ of order $p$ with $p \ge 2 s^2$. We also refer to \cite{CS3} where the arguments from \cite{CS02} are considerably simplified and to the overview in \cite[Chapter~16]{DP10}. The result in \cite{CS02} was extended to the $\mathcal{L}_p$ discrepancy for $1 \le p < \infty$ by Skriganov~\cite{Skr}.

On the other hand, it was shown by Proinov~\cite{pro85} that for an infinite sequence $\cS_s$ of points in $[0,1)^s$ there is a constant $c'_s > 0$ such that
\begin{equation*}
\mathcal{L}_{2,N}(\cS_s) \ge c'_s \frac{(\log N)^{s/2}}{N}
\end{equation*}
for infinitely many values of $N$. This lower bound is known to be best possible in dimension $s=1$. One-dimensional infinite sequences whose $\LL_2$ discrepancy satisfies a bound of order $\sqrt{\log N}/N$ for every $N \ge 2$ were given in, e.g. \cite{chafa,g96,lp,pro85,pg}. These constructions are mainly based on the symmetrization of sequences (also called reflection principle). On the other hand, although it was widely believed that Proinov's lower bound is also best possible for arbitrary dimensions $s$, so far there was no proof for this assertion.

\subsection{The main results}

In this paper we prove two main results: We provide for the first time explicit constructions of {\it infinite} sequences in $[0,1)^s$ for which the first $N \ge 2$ points achieve a  $\LL_2$ discrepancy of order $(\log N)^ {s/2}/N$ for arbitrary $s \in \NN$. This result is best possible by the lower bound of Proinov~\cite{pro85}.

Furthermore, for any integer $N \ge 2$ and any dimension $s \in \NN$, we give an explicit construction of a finite point set of $N$ elements in the $s$-dimensional unit cube with the optimal rate of convergence for the $\LL_2$ discrepancy in the sense of the lower bound of Roth. Our construction is completely different from the construction of Chen and Skriganov \cite{CS02}. In contrast to \cite{CS02} where the construction uses a finite field $\FF_p$ with $p \ge 2s^2$ our method is, independent of the dimension $s$, based on the finite field $\mathbb{F}_2$ of order two. Furthermore, our result does not use the Davenport reflection principle \cite{dav} and also does not use the 'self-averaging' property from \cite{CS02}. Instead it is based on higher order digital nets and sequences from \cite{D07, D08}.

In our proofs we do not keep track of constants which depend only on the dimension $s$ since they are significantly larger than the constants obtained in \cite{DP05b}. Therefore, in the following, we write $A(N,s) \ll_s B(N,s)$ if there is a constant $c_s > 0$ which depends only on $s$ (and not on $N$ or $m$ through $N=2^m$) such that $A(N,s) \le c_s B(N,s)$.

\begin{theorem}\label{thm1}
For any $s \in \NN$ one can explicitly construct an infinite sequence $\cS_s$ of points in $[0,1)^s$ such that for all $N \ge 2$ we have
\begin{equation*}
\mathcal{L}_{2,N}(\cS_s) \ll_s  \frac{(\log N)^{(s-1)/2}}{N} \sqrt{S(N)} \ll_s \frac{(\log N)^{s/2}}{N},
\end{equation*}
where $S(N)$ is the sum-of-digits function of $N$ in base 2 representation, i.e. if $N = 2^{m_1} + 2^{m_2} + \cdots + 2^{m_r}$ with $m_1 > m_2 > \cdots > m_r \ge 0$, then $S(N)=r$. Obviously, we have $S(N) \le 1+(\log N)/(\log 2)$ for all $N \in \NN$.
\end{theorem}

\begin{remark}\rm
It follows from \cite[Corollary~3]{LP05} that for any $\varepsilon >0$ we have $$\lim_{M \rightarrow \infty} \frac{1}{M}\left|\left\{0 \le N< M\, : \, (1-\varepsilon) \frac{\log M}{2 \log 2} < S(N) < (1+\varepsilon) \frac{\log M}{2 \log 2}\right\}\right|=1.$$ Hence the density of $N \in \NN$ for which $S(N)$ is at least of order $\log N$ is equal to one. More precise results on the distribution of the sum-of-digits function can be obtained, e.g., from \cite{BaKa,Man}.
\end{remark}

The above construction can also be used to obtain the following result for finite point sets, which was first shown in \cite{CS02} by a different construction.
\begin{corollary}\label{thm0}
For any $s \in \NN$ and any integer $N \ge 2$ one can explicitly construct a point set $\cP_{N,s}$ consisting of $N$ elements in $[0,1)^s$ such that $$\mathcal{L}_{2,N}(\cP_{N,s}) \ll_s  \frac{(\log N)^{(s-1)/2}}{N}.$$
\end{corollary}

\subsection{Explicit constructions of sequences and point sets}\label{sec_construction}

We now present explicit constructions of sequences and point sets satisfying Theorem~\ref{thm1} and Corollary~\ref{thm0}.

The construction of sequences $\cS_s=(\boldsymbol{x}_0,\boldsymbol{x}_1,\ldots)$ in $[0,1)^s$ satisfying Theorem~\ref{thm1} was introduced in \cite{D07,D08} and is based on linear algebra over the finite field $\mathbb{F}_2$ of order $2$ (we identify $\FF_2$ with the set $\{0,1\}$ equipped with the arithmetic operations modulo 2).

First we need to recall the definition of digital nets according to Niederreiter \cite{nie87,nie92}: For $m,p \in \NN$ with $p \ge m$ let $C_1,\ldots, C_s \in \mathbb{F}_2^{p \times m}$ be $p \times m$ matrices over $\mathbb{F}_2$. For $n \in \{0,\ldots ,2^m-1\}$ with binary expansion $n = n_0 + n_1 2 + \cdots + n_{m-1} 2^{m-1}$ we define the binary digit vector $\vec{n}$ as $\vec{n} = (n_0, n_1, \ldots, n_{m-1})^\top \in \mathbb{F}_2^{m}$ (the symbol $\top$ means the transpose of a vector or a matrix). Then compute
\begin{equation*}
C_j \vec{n} =:(x_{j,n,1}, x_{j,n,2},\ldots,x_{j,n,p})^\top \quad \mbox{for } j = 1,\ldots, s,
\end{equation*}
where the matrix vector product is evaluated over $\FF_2$, and put
\begin{equation*}
x_{j,n} = x_{j,n,1} 2^{-1} + x_{j,n,2} 2^{-2} + \cdots + x_{j,n,p} 2^{-p} \in \QQ(2^p).
\end{equation*}
The $n$th point $\boldsymbol{x}_n$ of the net $\cP_{2^m,s}$ is given by $\boldsymbol{x}_n = (x_{1,n}, \ldots, x_{s,n})$. A net $\cP_{2^m,s}$ constructed this way is called a {\it digital net (over $\FF_2$) with generating matrices} $C_1,\ldots,C_s$. Note that a digital net consists of $2^m$ elements in $\QQ(2^p)^s$.

We also recall the definition of digital sequences according to Niederreiter \cite{nie87,nie92}, which are infinite versions of digital nets. Let $C_1,\ldots, C_s \in \mathbb{F}_2^{\mathbb{N} \times \mathbb{N}}$ be $\mathbb{N} \times \mathbb{N}$ matrices over $\mathbb{F}_2$. For $C_j = (c_{j,k,\ell})_{k, \ell \in \mathbb{N}}$ we assume that for each $\ell \in \mathbb{N}$ there exists a $K(\ell) \in \mathbb{N}$ such that $c_{j,k,\ell} = 0$ for all $k > K(\ell)$. For $n \in \NN_0$, where $\NN_0=\NN \cup \{0\}$, with binary expansion $n = n_0 + n_1 2 + \cdots + n_{m-1} 2^{m-1} \in \mathbb{N}_0$,  we define the infinite dyadic digit vector of $n$ by $\vec{n} = (n_0, n_1, \ldots, n_{m-1}, 0, 0, \ldots )^\top \in \mathbb{F}_2^{\mathbb{N}}$. Then compute
\begin{equation*}
C_j \vec{n}=:(x_{j,n,1}, x_{j,n,2},\ldots)^\top \quad \mbox{for } j = 1,\ldots, s,
\end{equation*}
where the matrix vector product is evaluated over $\FF_2$, and put
\begin{equation*}
x_{j,n} = x_{j,n,1} 2^{-1} + x_{j,n,2} 2^{-2} + \cdots \in [0,1).
\end{equation*}
The $n$th point $\boldsymbol{x}_n$ of the sequence $\cS_s$ is given by $\boldsymbol{x}_n = (x_{1,n}, \ldots, x_{s,n})$. A sequence $\cS_s$ constructed this way is called a {\it digital sequence (over $\FF_2$) with generating matrices} $C_1,\ldots,C_s$. Note that since $c_{j,k,\ell} = 0$ for all $k$ large enough, the numbers $x_{j,n}$ are always dyadic rationals. (We call $x\in [0,1)$ a dyadic rational if it can be written in a finite base $2$ expansion.)

Explicit constructions of suitable generating matrices $C_1,\ldots, C_s$ over $\mathbb{F}_2$ were obtained by Sobol'~\cite{sob67}, Niederreiter~\cite{nie87,nie92}, Niederreiter-Xing~\cite{NX96} and others (see \cite[Chapter~8]{DP10} for an overview). Any of these constructions is sufficient for our purpose, however, for completeness, we briefly describe a special case of Tezuka's construction~\cite{Tez}, which is a generalization of Sobol's construction~\cite{sob67} and Niederreiter's construction~\cite{nie87} of the generating matrices. 

We explain how to construct the entries $c_{j,k,\ell} \in \mathbb{F}_2$ of the generator matrices $C_j = (c_{j,k,\ell})_{k,\ell \ge 1}$ for $j=1,2,\ldots,s$. To this end choose the polynomials $p_1=x$ and $p_j \in \mathbb{F}_2[x]$ for $j =2,\ldots,s$ to be the $(j-1)$th primitive polynomial in a list of primitive polynomials over $\mathbb{F}_2$ that is sorted in increasing order according to their degree $e_j = \deg(p_j)$, that is, $e_2 \le e_3 \le \cdots \le e_{s-1}$ (the ordering of polynomials with the same degree is irrelevant). We also put $e_1=\deg(x)=1$. (We point out that Niederreiter~\cite{nie87} uses irreducible polynomials instead of primitive polynomials.)

Let $j \in \{1,\ldots,s\}$ and $k \in \NN$. Take $i-1$ and $z$ to be respectively the main term and remainder when we divide $k-1$ by $e_j$, so that   $k-1  = (i-1) e_j + z$, with $0 \le z < e_j$. Now consider the Laurent series expansion
\begin{equation*}
\frac{x^{e_j-z-1}}{p_j(x)^i} = \sum_{\ell =1}^\infty a_\ell(i,j,z) x^{-\ell} \in \mathbb{F}_2((x^{-1})).
\end{equation*}
For $\ell \in \mathbb{N}$ we set
\begin{equation}\label{def_sob_mat}
c_{j,k,\ell} = a_\ell(i,j,z).
\end{equation}
Every digital sequence with generating matrices $C_j = (c_{j,k,\ell})_{k,\ell \ge 1}$ for $j=1,2,\ldots,s$ found in this way is a special instance of a Sobol' sequence which in turn is a special instance of so-called generalized Niederreiter sequences (see \cite[Eq. (3)]{Tez}). Note that in the construction above we always have $c_{j,k,\ell}=0$ for all $k > \ell$.

Note that generalized Niederreiter sequences (as are Sobol's and Niederreiter's sequences) are digital $(t,s)$-sequences with
\begin{equation}\label{t_digseq}
t = \sum_{j=1}^s (e_j-1).
\end{equation}
See \cite[Lemma~4]{Tez} for details.

To obtain a sequence which satisfies Theorem~\ref{thm1} we need the following definition.

\begin{definition}\rm
For $\alpha \in \NN$ the {\it digit interlacing composition} (with interlacing factor $\alpha$) is defined by
\begin{eqnarray*}
\mathscr{D}_\alpha: [0,1)^{\alpha} & \to & [0,1) \\
(x_1,\ldots, x_{\alpha}) &\mapsto & \sum_{a=1}^\infty \sum_{r=1}^\alpha
\xi_{r,a} 2^{-r - (a-1) \alpha},
\end{eqnarray*}
where $x_r \in [0,1)$ has dyadic expansion of the form $x_r = \xi_{r,1} 2^{-1} + \xi_{r,2} 2^{-2} + \cdots$ for $1 \le r \le \alpha$. We also define this function for vectors by setting
\begin{eqnarray*}
\mathscr{D}_\alpha^s: [0,1)^{\alpha s} & \to & [0,1)^s \\
(x_1,\ldots, x_{\alpha s}) &\mapsto & (\mathscr{D}_\alpha(x_1,\ldots, x_\alpha),  \ldots, \mathscr{D}_\alpha(x_{(s-1) \alpha +1},\ldots, x_{\alpha s})),
\end{eqnarray*}
for point sets $\cP_{N,\alpha s} = \{\bsx_0,\bsx_1, \ldots, \bsx_{N-1}\} \subseteq [0,1)^{\alpha s}$ by setting
\begin{equation*}
\mathscr{D}_\alpha^s(\cP_{N,\alpha s}) = \{\mathscr{D}_\alpha^s(\bsx_0), \mathscr{D}_\alpha^s(\bsx_1), \ldots, \mathscr{D}_\alpha^s(\bsx_{N-1})\}\subseteq[0,1)^s
\end{equation*}
and for sequences $\cS_{\alpha s} = (\bsx_0, \bsx_1, \ldots)$ with $\bsx_n \in [0,1)^{\alpha s}$ by setting
\begin{equation*}
\mathscr{D}_\alpha^s(\cS_{\alpha s}) = (\mathscr{D}_\alpha^s(\bsx_0), \mathscr{D}_\alpha^s(\bsx_1), \ldots).
\end{equation*}
\end{definition}

We comment here that the interlacing can also be applied to the generating matrices $C_1,\ldots, C_{\alpha s}$ directly as described in \cite[Section~4.4]{D08}. This is done in the following way. Let $C_1, \ldots, C_{\alpha s}$ be generating matrices of a digital net or sequence and let $\vec{c}_{j,k}$ denote the $k$th row of $C_j$. We define matrices $E_1,\ldots, E_s$, where the $k$th row of $E_j$ is given by $\vec{e}_{j,k}$, in the following way. For all $1 \le j \le s$, $u \ge 0$ and $1 \le v \le \alpha$ let
\begin{equation*}
\vec{e}_{j,u\alpha + v} = \vec{c}_{(j-1) \alpha + v, u+1}
\end{equation*}
If $C_1, \ldots, C_{\alpha s}$ are the generating matrices of a digital net $\cP_{N, \alpha s}$ or digital sequence $\cS{\alpha s}$ respectively, then the matrices $E_1,\ldots, E_s$ defined above, are the generating matrices of $\mathscr{D}_\alpha^s(\cP_{N,\alpha s})$ or $\mathscr{D}_\alpha^s(\cS_{\alpha s})$ respectively. Thus one can also obtain generating matrices $E_1,\ldots, E_s \in \mathbb{F}_2^{\mathbb{N} \times \mathbb{N}}$ which generate a digital sequence satisfying Theorem~\ref{thm1}.

Above we assumed that $c_{j,k,\ell}=0$ for all $k > K(\ell)$. Let $E_j = (e_{j,k,\ell})_{k, \ell \in \mathbb{N}}$. Then the interlacing construction yields that $e_{j,k,\ell} = 0$ for all $k > \alpha K(\ell)$, where $\alpha$ is the interlacing factor.

We shall show that the sequence $\mathscr{D}_5^s(\cS_{5s})$, where $\cS_{5s}$ is a digital sequence in dimension $5s$ constructed, for example, according to Sobol' as presented above, satisfies the bounds in Theorem~\ref{thm1}.\\

To construct finite point sets for any integer $N \ge 2$ we proceed in the following way. Let $m \in \NN$ be such that $2^{m-1} <N \le 2^m$ and let $\bsx_0, \bsx_1,\ldots, \bsx_{2^m-1} \in [0,1]^{3s-1}$ be the first $2^m$ points from the Sobol' or Niederreiter sequence in dimension $3s-1$ as introduced above with $p_1=x$ and $p_2= 1+ x$. Let $\bsx_n = (x_{1,n},\ldots, x_{3s-1,n})$ and define $\bsy_n = (n 2^{-m}, x_{1,n},\ldots, x_{3s-1,n}) \in [0,1)^{3s}$. Let now $\cP_{2^m,s} = \{\mathscr{D}_3(\bsy_0), \mathscr{D}_3(\bsy_1), \ldots, \mathscr{D}_3(\bsy_{2^m})\}$. To obtain a point set consisting of $N$ points we use a propagation rule introduced in \cite{CS02} (see also \cite[p.~512]{DP10}): The subset
$$\widetilde{\cP}_{N,s}:=\cP_{2^m,s} \cap
\left(\left[0,\frac{N}{2^m}\right) \times [0,1)^{s-1}\right)$$
contains exactly $N$ points. Then we define the point set
\begin{equation}\label{pspropcs}
\cP_{N,s}:=\left\{\left(\frac{2^m}{N} x_1,x_2,\ldots,x_s\right)\, :
\, (x_1,x_2,\ldots,x_s) \in \widetilde{\cP}_{N,s}\right\}.
\end{equation}
We will show that $\cP_{N,s}$ satisfies the bound in Corollary~\ref{thm0}. We remark that Chen and Skriganov~\cite{CS02} applied the same propagation rule but to a different point set.

\subsection{The general construction principle}\label{sec_gen_princ}

Our approach is based on higher order digital nets and sequences constructed explicitly in \cite{D07,D08}. We state here simplified versions of their definitions that are sufficient for our purpose. For $p \in \NN$ let $\QQ(2^p):=\left\{0,\frac{1}{2^p},\frac{2}{2^p},\ldots,\frac{2^p -1}{2^p}\right\}$.

The distribution quality of digital nets and sequences depends on the choice of the respective generating matrices. In the following definitions we put some restrictions on $C_1,\ldots ,C_s$ with the aim to quantify the quality of equidistribution of the digital net or sequence.

\begin{definition}\rm\label{def_net}
Let $m, p, \alpha \in \NN$ with $p \ge \alpha m$ and let $t$ be an integer such that $0 \le t \le \alpha m$. Let $C_1,\ldots, C_s \in \FF_2^{p \times m}$ with $C_j = (\vec{c}_{j,1},
\ldots, \vec{c}_{j, p})^\top$, i.e., $\vec{c}_{j,i} \in \FF_2^m$ is the $i$th row vector of the matrix $C_j$. If for all $1 \le i_{j,\nu_j} < \cdots <
i_{j,1} \le p$ with $$\sum_{j = 1}^s \sum_{l=1}^{\min(\nu_j,\alpha)} i_{j,l}  \le
\alpha m - t$$ the vectors
$$\vec{c}_{1,i_{1,\nu_1}}, \ldots, \vec{c}_{1,i_{1,1}}, \ldots,
\vec{c}_{s,i_{s,\nu_s}}, \ldots, \vec{c}_{s,i_{s,1}}$$ are linearly independent
over $\FF_2$, then the digital net with generating matrices
$C_1,\ldots, C_s$ is called an {\it order $\alpha$ digital $(t,m,s)$-net over $\FF_2$}.
\end{definition}

Next we consider digital sequences for which the initial segments are order $\alpha$ digital $(t, m,s)$-nets over $\FF_2$:

\begin{definition}\rm\label{def_seq}
Let $\alpha \in \NN$ and let $t \ge 0$ be an integer. Let $C_1,\ldots, C_s \in \FF_2^{\mathbb{N} \times \mathbb{N}}$ and let $C_{j, \alpha m \times m}$ denote the left upper $\alpha m \times m$ submatrix of  $C_j$. If for all $m > t/\alpha$ the matrices $C_{1, \alpha m \times m},\ldots, C_{s, \alpha m \times m}$ generate an order $\alpha$ digital $(t, m,s)$-net over $\FF_2$, then the digital sequence with generating matrices $C_1,\ldots, C_s$ is called an {\it order $\alpha$ digital $(t,s)$-sequence over $\FF_2$}.
\end{definition}

From Definition~\ref{def_net} it is clear the if $\cP_{2^m,s}$ is an order $\alpha$ digital $(t,m,s)$-net, then for any $t \le t' \le \alpha m$, $\cP_{2^m,s}$ is also an order $\alpha$ digital $(t',m,s)$-net. An analogue result also applies to higher order digital sequences.

From \cite[Theorem~4.11 and Theorem~4.12]{D07} (where we set $\alpha = d$) we obtain the following result:
\begin{proposition}
If $\cS_{\alpha s}$ is an order 1 digital $(t',\alpha s)$-sequence over $\FF_2$, then $\mathscr{D}_\alpha^s(\cS_{\alpha s})$ is an order $\alpha$ digital $(t,s)$-sequence over $\FF_2$ with $$t = \alpha t' + s {\alpha \choose 2}.$$
\end{proposition}
For the construction based on Sobol's and Niederreiter's sequence introduced above we have \eqref{t_digseq} and therefore we obtain explicit constructions of order $\alpha$ digital $(t,s)$-sequences with
\begin{equation*}
t = \alpha \sum_{j=1}^s (e_j-1) + s {\alpha \choose 2}.
\end{equation*}
Note that for the construction introduced above we have $c_{j,k, \ell} = 0$ for all $k > \ell$. Using the interlacing construction we obtain generating matrices $E_1,\ldots, E_s$ with $E_j = (e_{j,k,\ell})_{k,\ell \in \mathbb{N}}$ and $e_{j,k,\ell} = 0$ for all $k > \alpha \ell$. Let $E_{j, \mathbb{N} \times m}$ denote the first $m$ columns of $E_{j}$. Then we obtain that the $k$th row of $E_{j, \mathbb{N} \times m}$ is the zero-vector for all $k > \alpha m$. This implies that the first $2^m$ points of the digital sequence with generating matrices $E_1,\ldots, E_s$ are the same as the points of the digital net with generating matrices $E_{1, \alpha m \times m}, \ldots, E_{s, \alpha m \times m}$. In particular this implies that all coordinates of all points are dyadic rationals. (For more general constructions of digital $(t,s)$-sequences a similar result holds, however we do not use this fact here.)

Note that a digital net can be an order $\alpha$ digital $(t,m,s)$-net over $\FF_2$ and at the same time an order $\alpha'$ digital $(t',m,s)$-net over $\
\FF_2$ for $\alpha'\not = \alpha$. This means that the quality parameter $t$ may depend on $\alpha$. If necessary we write $t(\alpha)$ instead of $t$ for the quality parameter of an order $\alpha$ digital $(t(\alpha),m, s)$-net. The same holds for digital sequences. In particular \cite[Theorem~4.10]{D08} implies that an order $\alpha$ digital $(t,m,s)$-net is an order $\alpha'$ digital $(t',m,s)$-net for all $1 \le \alpha' \le \alpha$ with
\begin{equation}\label{eq_t_tprime}
t' = \lceil t \alpha'/\alpha \rceil \le t.
\end{equation}
The same result applies to order $\alpha$ digital $(t,s)$-sequences which are also order $1\le \alpha' \le \alpha$ digital $(t',s)$-sequences with $t'$ as above. In other words, $t(\alpha') = \lceil t(\alpha) \alpha'/\alpha \rceil$ for all $1 \le \alpha' \le \alpha$. More information can be found in \cite[Chapter~15]{DP10}.

We will show that every order $\alpha$ digital $(t,s)$-sequence over $\FF_2$ with $\alpha \ge 5$ satisfies the requirements of Theorem~\ref{thm1}.

\subsection{Geometric properties of (higher order) digital nets}

We give a geometric interpretation of the digital nets introduced above. For $\alpha=1$ they go back to Niederreiter \cite{nie87,nie92}. The condition in Definition~\ref{def_net} says that so-called dyadic elementary boxes of the form $$\prod_{j=1}^s \left[\frac{a_j}{2^{d_j}}, \frac{a_j+1}{2^{d_j}}\right)$$ with integers $d_j \ge 0$, $d_1+\cdots + d_s = m-t$ and integers $0 \le a_j < 2^{d_j}$, contain $b^t$ points of the net, which is the fair portion of points of the net with respect to the volume of the box. Thus smaller values of the so-called quality parameter $t$ imply stronger equidistribution properties of a net. For more information see \cite[Theorem~4.28]{nie92} or \cite[Theorem~4.52]{DP10}.

The more general definition for $\alpha >1$ goes back to Dick \cite{D07,D08}. Rather than considering boxes containing the right portion of points as for the case $\alpha=1$, here one considers unions of such boxes. To give the geometric interpretation, we define for $\nu \in \mathbb{N}_0$, $a_1>a_2 >\cdots > a_\nu \ge -\nu + 1$ and $\kappa_1, \kappa_2, \ldots, \kappa_\nu \in \{0, 1\}$ the union of intervals
\begin{eqnarray*}
\lefteqn{J_\alpha(a_1, \ldots, a_{\nu}, \kappa_1,\ldots, \kappa_\nu)}  \\ & = & \left\{x \in [0,1): x = \sum_{d=1}^\infty \xi_d 2^{-d} \mbox{ with } \xi_{a_i} = \kappa_i \mbox{ for } i = 1, \ldots, \nu \right\},
\end{eqnarray*}
where we set $J=[0,1)$ for $\nu=0$, where $a_i \in \{-\nu + 1, -\nu + 2, \ldots, 0\}$ does not yield any restriction and where we always use the finite expansion of $x$ for dyadic rationals. For instance we have $J_2(0, -1, 0, 0) = [0,1)$, $J_2(1,0,0,0) = [0,1/2)$ and $J_2(3,1,1,1) = [5/8, 6/8) \cup [7/8, 1)$. Let $1_J(x)$  denote the indicator function of a set $J$ (which is 1 for $x \in J$ and 0 otherwise). Then an order $\alpha$ digital $(t,m,s)$-net satisfies
\begin{equation*}
\sum_{n=0}^{2^m-1} 1_J(\bsx_n) = \mathrm{Volume}(J),
\end{equation*}
for all $J$ of the form
\begin{equation*}
\prod_{j=1}^s J_\alpha(a_{1,j}, \ldots, a_{\nu_j,j}, \kappa_{1,j}, \ldots, \kappa_{\nu_j,j})
\end{equation*}
for all $\kappa_{r,j} \in \{0,1\}$ for all $1 \le r \le \nu_j$ and $1 \le j \le s$ and all $a_{j,1} > a_{j,2} > \cdots > a_{j,\nu_j} > -\nu_j + 1$ with
\begin{equation*}
\sum_{j=1}^s \sum_{r=1}^{\min\{\nu_j, \alpha\}} \max\{a_{j,r}, 0\} \le \alpha m - t.
\end{equation*}
Thus higher order digital nets do not only contain the correct proportion of points for elementary dyadic intervals, but also for certain unions of disjoint dyadic intervals. Thus higher order digital nets have an additional structure which classical digital nets do not necessarily have.

\section{Walsh series representation of the squared $\LL_2$ discrepancy}

As an important tool in our analysis we use a Walsh series representation of the $\LL_2$ discrepancy. This representation will be deduced within this chapter.

\subsection{Walsh functions}\label{sect_walsh}

We introduce Walsh functions in base $2$ (see \cite{chrest, fine, walsh}), which will be the main tool in our analysis of the $\LL_2$ discrepancy. We recall that $\mathbb{N}_0=\NN \cup \{0\}$.

For $k \in \NN_0$ the $k$th  Walsh function
$\wal_{k}:[0,1) \rightarrow \{-1,1\}$ is defined in the following way: let $k$ have base $2$ representation
\[
   k = \kappa_{a-1} 2^{a-1} + \cdots + \kappa_1 2 + \kappa_0,
\]
with $\kappa_i \in \{0,1\}$, and let $x \in [0,1)$ have base $2$ representation $$x =
\frac{\xi_1}{2}+\frac{\xi_2}{2^2}+\cdots $$ with $\xi_i \in \{0,1\}$ (unique in the sense that
infinitely many of the $\xi_i$ must be zero), then
\[
  \wal_{k}(x) := (-1)^{\xi_1 \kappa_0 + \cdots + \xi_a \kappa_{a-1}}.
\]

For dimension $s \geq 2$, vectors $\bsk = (k_1, \ldots, k_s) \in \nat_0^s$ and $\bsx =
(x_1,\ldots,x_s) \in [0,1)^s$ we write
\[
   \wal_{\bsk}(\bsx) := \prod_{j=1}^s\wal_{k_j}(x_j).
\]
A summary of properties of Walsh functions can also be found in \cite[Appendix~A]{DP10}. See also \cite{CS00} for Walsh functions in the context of discrepancy theory, \cite{LT94} for Walsh functions in the related context of numerical integration in \cite{LT94}, or \cite{Tez2} in the related context of pseudo random number generation.\\

We report on a relation between Walsh functions and digital nets over $\FF_2$ which will be useful for our analysis. Before we do so we need to introduce some further notation. By $\oplus$ we denote the digit-wise
addition modulo $2$, i.e., for real numbers $x, y \ge 0$ with dyadic expansion $x = \sum_{i=w}^{\infty}
\frac{\xi_i}{2^i}$ and $y = \sum_{i=w}^{\infty} \frac{\eta_i}{2^i}$ with $w \in \mathbb{Z}$ and $\xi_i\not=1$ for infinitely many $i$ and $\eta_j \not=1$ for infinitely many $j$, we put
$$ x \oplus y := \sum_{i=w}^{\infty } \frac{\zeta_i}{2^i}, \; \; \;{\rm where}
\; \; \; \zeta_i := \xi_i + \eta_i \mypmod{2}.$$ For vectors $\bsx, \bsy \in [0,1)^s$ we set $\bsx \oplus \bsy = (x_1 \oplus y_1, \ldots, x_s \oplus y_s)$. Note that e.g. for $x = 2^{-1} + 2^{-3} + 2^{-5} + \cdots$ and $y = 2^{-2} + 2^{-4} + 2^{-6} + \cdots$ we have $x \oplus y = 2^{-1} + 2^{-2} + 2^{-3} + \cdots = 1$, see \cite[Section~2]{fine}. Thus $x \oplus y$ is a dyadic rational which is not defined via its finite expansion. However, in this paper, we only use $\oplus$ in conjunction with dyadic rationals $x$ and $y$ for which we assume that $x$ and $y$ are given by their finite expansion. Therefore, in this paper, $x \oplus y$ will always be a dyadic rational defined via its finite expansion.

It can be shown (see \cite[Lemma~4.72]{DP10}) that any digital net $\cP_{2^m,s}$ is a subgroup of $([0,1)^s,\oplus)$. Since for any $\bsx_h, \bsx_j \in \cP_{2^m,s}$ and any $\bsk \in \NN_0^s$ we have $$\wal_{\bsk}(\bsx_h \oplus \bsx_j)=\wal_{\bsk}(\bsx_h)\wal_{\bsk}(\bsx_j)$$ it follows that $\wal_{\bsk}$ is a character of the group $(\cP_{2^m,s},\oplus)$. Hence, for any digital net $\cP_{2^m,s}$ with generating matrices $C_1,\ldots,C_s \in \FF_2^{p \times m}$ and any $\bsk=(k_1,\ldots,k_s) \in \NN_0^s$ it follows that
\begin{equation}\label{charprop}
\sum_{h=0}^{2^m-1} \wal_{\bsk}(\bsx_h)=\left\{
\begin{array}{ll}
2^m & \mbox{ if } C_1^{\top} \vec{k}_1+\cdots +C_s^{\top} \vec{k}_s=\vec{0},\\
0 & \mbox{ otherwise},
\end{array}\right.
\end{equation}
where for $k_j \in \mathbb{N}_0$ with dyadic expansion $k_j = \kappa_{j,0} + \kappa_{j,1} 2 + \cdots + \kappa_{j,a-1} 2^{a-1}$ we set $\vec{k}_j = (\kappa_{j,0}, \kappa_{j,1}, \ldots, \kappa_{j, p-1})^\top$ with $\kappa_{j,a} = \kappa_{j,a+1} = \cdots = \kappa_{j,p-1} = 0$ for $a < p$. For a proof of this fact we refer to \cite[Lemma~4.75]{D10} (therein only $p=m$ was considered, but only minor modifications are required to obtain a proof of \ref{charprop}). We will call this relation the {\it character property} of digital nets.

\subsection{The Walsh series expansion of the $\LL_2$ discrepancy}

The squared $\LL_2$ discrepancy of a point set $\cP_{N,s} =
\{\bsx_0,\ldots, \bsx_{N-1}\}$ can be viewed as a function
of $\{\bsx_0,\ldots, \bsx_{N-1}\}$, i.e. a function of $N s$
variables:
$$\LL^2_{2,N}(\cP_{N,s})=\LL_{2,N}^2(\{\bsx_0,\ldots, \bsx_{N-1}\}).$$ To obtain its Walsh series expansion, we use the following well known formula of Warnock \cite{war} (see also  \cite[Proposition~2.15]{DP10}).
\begin{proposition}\label{pr1}
Let $\cP_{N,s}=\{\bsx_0,\ldots ,\bsx_{N-1}\}$ be a point set in $[0,1)^s$.
Then we have
\begin{eqnarray*}\label{L2-dis}
\LL_{2,N}^2(\cP_{N,s}) =  \frac{1}{3^{s}}-\frac{2}{N}\sum_{n=0}^{N-1}\prod_{j=1}^s
\frac{1-x_{n,j}^2}{2}+\int_{[0,1]^s} \left( \frac{A_N([\bszero,\bst),
\cP_{N,s})}{N}\right)^2 \,\rd \bst,
\end{eqnarray*}
where $x_{n,j}$ is the $j$th component of the point $\bsx_n$.
\end{proposition}

We need the Walsh series expansion of the indicator function $1_{[0,t)}(x)$ (which is $1$ for $0 \le x < t$ and $0$ otherwise), which was first given by Fine~\cite{fine} and which is nowadays well known. To state this expansion we need a weight function $\mu$ defined for non-negative integers. Put $\mu(0)=0$ and for $k \in \NN$ with base $2$ representation
$k = \kappa_0 + \kappa_1 2 + \cdots + \kappa_{a-2} 2^{a-2} + 2^{a-1}$ with $\kappa_i \in \{0,1\}$ put $\mu(k):= a$.

Then for $x \in [0,1]$ the Walsh series expansion of $1_{[0,t)}(x)$ is given as
\begin{eqnarray*}\label{walshcoeffind}
1_{[0,t)}(x) & \simeq & 1-x    \\ & &+ \sum_{k=1}^\infty \frac{1}{2^{\mu(k)+1}}
\left(\sum_{r=1}^\infty \frac{1}{2^r} \wal_{k \oplus 2^{r+ \mu(k)-1}}(x) - \wal_{k
\oplus 2^{\mu(k)-1}}(x) \right) \wal_k(t).\nonumber
\end{eqnarray*}

Using Parseval's identity we therefore obtain
\begin{eqnarray*}
\lefteqn{\int_0^1 1_{[0,t)}(x) 1_{[0,t)}(y) \, \rd t}\\  & = &
(1-x)(1-y) + \sum_{k=1}^\infty \frac{1}{2^{2\mu(k)+2}} \left(\wal_{k \oplus
2^{\mu(k)-1}}(x) - \sum_{r=1}^\infty \frac{1}{2^r} \wal_{k \oplus 2^{r+\mu(k)-1}}(x)
\right)  \\ & &\hspace{4cm} \times\left(\wal_{k \oplus 2^{\mu(k)-1}}(y) -
\sum_{r=1}^\infty \frac{1}{2^r} \wal_{k \oplus 2^{r+\mu(k)-1}}(y)\right).
\end{eqnarray*}

Using the fact that $A_N([\bszero,\bst),\cP_{N,s}) = \sum_{n=0}^{N-1} \prod_{j=1}^s
1_{[0,t_j)}(x_{n,j})$ it follows that
\begin{eqnarray*}
\int_{[0,1]^s} \left( \frac{A_N([\bszero,\bst), \cP_{N,s})}{N}\right)^2
\,\rd \bst & = &  \frac{1}{N^2} \sum_{n,m=0}^{N-1} \prod_{j=1}^s \int_0^1
1_{[0,t_j)}(x_{n,j}) 1_{[0,t_j)}(x_{m,j}) \, \rd t_j.
\end{eqnarray*}
Combining the last two equations we obtain
\begin{eqnarray}\label{walshseriesA}
\lefteqn{\int_{[0,1]^s} \left( \frac{A_N([\bszero,\bst),
\cP_{N,s})}{N}\right)^2 \,\rd \bst } \nonumber \\ & = & \frac{1}{N^2}
\sum_{n,m=0}^{N-1} \prod_{j=1}^s \Bigg[(1-x_{n,j})(1-x_{m,j})
\nonumber \\ &&  + \sum_{k=1}^\infty \frac{1}{2^{2\mu(k)+2}} \left(\wal_{k
\oplus 2^{\mu(k)-1}}(x_{n,j}) - \sum_{r=1}^\infty \frac{1}{2^r} \wal_{k \oplus
2^{r+\mu(k)-1}}(x_{n,j}) \right) \nonumber \\ &&\hspace{1cm}\times \left(\wal_{k
\oplus 2^{\mu(k)-1}}(x_{m,j}) - \sum_{r=1}^\infty \frac{1}{2^r} \wal_{k \oplus
2^{r+\mu(k)-1}}(x_{m,j}) \right) \Bigg].
\end{eqnarray}

The Walsh series representation of $(1-x_{n,j})(1-x_{m,j})$ can
easily be found. For example it was shown in \cite[Lemma~A.22]{DP10} that
\begin{equation}\label{walshseriesx}
x-\frac{1}{2} = - \sum_{a=1}^\infty \frac{1}{2^{a+1}} \wal_{2^{a-1}}(x).
\end{equation}
Using \eqref{walshseriesA} together with the last equality we obtain
the Walsh series representation of $\int_{[0,1]^s} \left(
\frac{A_N([\bszero,\bst), \cP_{N,s})}{N}\right)^2 \rd \bst$.

Using again \eqref{walshseriesx} and Proposition~\ref{pr1} we can now obtain the Walsh series
expansion of the squared $\LL_2$ discrepancy, which is given by
\begin{eqnarray}\label{l2fo1}
\lefteqn{\LL_{2,N}^2(\cP_{N,s})}  \\ &= & \frac{1}{3^s} -
\frac{2}{N} \sum_{n=0}^{N-1}  \prod_{j=1}^s \Bigg(\frac{1}{3} +
\sum_{a=1}^\infty \frac{1}{2^{a+2}} \wal_{2^{a-1}}(x_{n,j})  \nonumber \\ && \hspace{4cm}- \sum_{1
\le a < a'} \frac{1}{2^{a+a'+2}} \wal_{2^{a-1}\oplus2^{a'-1}}(x_{n,j}) \Bigg) \nonumber \\
&& + \frac{1}{N^2} \sum_{n,m=0}^{N-1} \prod_{j=1}^s
\Bigg[\left(\frac{1}{2} + \sum_{a=1}^\infty \frac{1}{2^{a+1}}
\wal_{2^{a-1}}(x_{n,j})\right)\left(\frac{1}{2} + \sum_{a=1}^\infty \frac{1}{2^{a+1}} \wal_{2^{a-1}}(x_{m,j})\right) \nonumber \\
&& \hspace{1cm}+ \sum_{k=1}^\infty \frac{1}{2^{2\mu(k)+2}} \left( \wal_{k \oplus
2^{\mu(k)-1}}(x_{n,j}) - \sum_{r=1}^\infty \frac{1}{2^r} \wal_{k \oplus
2^{r+\mu(k)-1}}(x_{n,j})\right) \nonumber \\ && \hspace{3cm}\times\left(\wal_{k
\oplus 2^{\mu(k)-1}}(x_{m,j}) - \sum_{r=1}^\infty \frac{1}{2^r} \wal_{k \oplus
2^{r+\mu(k)-1}}(x_{m,j}) \right) \Bigg].\nonumber
\end{eqnarray}
The following lemma can now be obtained upon comparing coefficients.
\begin{lemma}\label{lem_r}
For any $\cP_{N,s}=\{\bsx_0,\ldots, \bsx_{N-1}\}$ in $[0,1)^s$ we obtain
\begin{eqnarray*}
\LL_{2,N}^2(\cP_{N,s}) &= & \frac{1}{3^s} - \frac{2}{N}
\sum_{n=0}^{N-1} \sum_{\bsk \in \NN_0^s} r(\bsk,\bszero)
\wal_{\bsk}(\bsx_n) \\ && + \frac{1}{N^2} \sum_{n,m=0}^{N-1}
\sum_{\bsk,\bsl \in \NN_0^s} r(\bsk,\bsl) \wal_{\bsk}(\bsx_n)
\wal_{\bsl}(\bsx_m),
\end{eqnarray*}
where $\bsk = (k_1,\ldots, k_s)$, $\bsl = (l_1,\ldots, l_s)$,
$r(\bsk,\bsl) = \prod_{j=1}^s r(k_j,l_j)$. Further we have $r(k,l) =
r(l,k)$ and for non-negative integers $0 \le l \le k$ with $k =
2^{a_1-1} + \cdots + 2^{a_v-1}$ with $a_1 > \cdots > a_v > 0$ and $l
= 2^{b_1-1} + \cdots + 2^{b_w-1}$ with $b_1
> \cdots > b_v > 0$ we have
\begin{equation*}
r(k,l) = \left\{\begin{array}{ll}
\frac{1}{3} & \mbox{if } k = l =
0, \\
\frac{1}{2^{a_1+2}} & \mbox{if } v = 1 \mbox{ and } l =0, \\
-\frac{1}{2^{a_1 + a_2+2}} & \mbox{if } v = 2 \mbox{ and } l = 0,\\
 -\frac{1}{2^{a_1 + a_2 +
2}} & \mbox{if } v = w+2
> 2 \mbox{ and }
a_3 = b_1, \ldots, a_v = b_{v-2}, \\
\frac{1}{3\cdot 4^{a_1}} & \mbox{if } k = l > 0, \\
\frac{1}{2^{a_1+b_1+2}} & \mbox{if } v = w, a_1 \neq b_1 \mbox{ and
} a_2 = b_2, \ldots, a_v = b_v, \\ 0 & \mbox{otherwise.}
\end{array} \right.
\end{equation*}
\end{lemma}

\begin{proof}
As already mentioned, the result follows from  \eqref{l2fo1} upon comparing coefficients. For instance we have
\begin{eqnarray*}
\frac{1}{3} + \sum_{a=1}^\infty \frac{1}{2^{a+2}} \wal_{2^{a-1}}(x_{n,j})- \sum_{1
\le a < a'} \frac{1}{2^{a+a'+2}} \wal_{2^{a-1}\oplus2^{a'-1}}(x_{n,j})\\ =   \sum_{k=0}^\infty r(k,0) \wal_k(x_{n,j})
\end{eqnarray*}
with $$r(k,0)=\left\{
\begin{array}{ll}
\frac{1}{3} & \mbox{ if } k=0,\\
\frac{1}{2^{a+2}} & \mbox{ if } k=2^{a-1},\\
-\frac{1}{2^{a+a'+2}} & \mbox{ if } k=2^{a-1} \oplus 2^{a'-1},\\
0 & \mbox{ in all other cases.}
\end{array}\right.$$
The result follows by checking all cases.
\end{proof}

We can simplify the above formula further. But first we recall what we mean by a digitally shifted digital net:

\begin{definition}\rm
Let $\cP_{2^m,s}=\{\bsx_0,\ldots, \bsx_{2^m-1}\}$ be a digital net over $\FF_2$ and let $\bssigma \in [0,1)^s$. Then we call the point set $\cP_{2^m,s}(\bssigma)=\{\bsx_0 \oplus \bssigma, \ldots, \bsx_{2^m-1} \oplus \bssigma\}$ a {\it digitally shifted digital net over $\FF_2$}.
\end{definition}

In this paper we will only consider digital shifts which are dyadic rationals. Since the points of a digital net are also dyadic rationals, the operation $\oplus$ is well defined.

\begin{lemma}\label{lem_l2formula}
We have:
\begin{itemize}
\item The squared $\LL_2$ discrepancy of a point set $\cP_{N,s}=\{\bsx_0,\ldots, \bsx_{N-1}\}$ in $[0,1)^s$ can be written as
\begin{equation*}
\LL^2_{2,N}(\cP_{N,s}) = \sum_{\bsk,\bsl \in \NN_0^s \setminus\{\bszero\}}
r(\bsk,\bsl) \frac{1}{N}\sum_{n=0}^{N-1} \wal_{\bsk}(\bsx_n)
\frac{1}{N} \sum_{m=0}^{N-1} \wal_{\bsl}(\bsx_{m}),
\end{equation*}
where the coefficients $r(\bsk,\bsl)$ are given as in Lemma~\ref{lem_r}.

\item If $\cP_{2^m,s}$ is a digital net over $\FF_2$ with generating matrices $C_1,\ldots, C_s \in \mathbb{F}_2^{p \times m}$ we have
\begin{equation*} \LL^2_{2,2^m}(\cP_{2^m,s}) =
\sum_{\bsk,\bsl \in \Dcal^\ast} r(\bsk,\bsl),
\end{equation*}
where $\Dcal^\ast=\Dcal\setminus \{\bszero\}$ and where $\Dcal$ is the so-called dual net given by $$\Dcal=\{(k_1,\ldots,s) \in \NN_0^s\, : \, C_1^{\top}\vec{k}_1+\cdots+C_s^{\top}\vec{k}_s=\vec{0}\},$$ where for $k \in \NN_0^s$ with base 2 expansion $k=\kappa_0+\kappa_1 2+\kappa_2 2^2+\cdots$ we put $\vec{k}=(\kappa_0,\ldots,\kappa_{p-1})^{\top}$.

\item If $\cP_{2^m,s}(\bssigma)$ is a digital net over $\FF_2$ digitally shifted by digital shift $\bssigma$ we have
\begin{equation*} \LL^2_{2,2^m}(\cP_{2^m,s}(\bssigma)) =
\sum_{\bsk,\bsl \in \Dcal^\ast} r(\bsk,\bsl) \wal_{\bsk}(\bssigma) \wal_{\bsl}(\bssigma),
\end{equation*}
where $\Dcal^\ast$ denotes the dual net excluding $\bszero$.
\end{itemize}
\end{lemma}

\begin{proof}
From $r(\bszero,\bszero)=3^{-s}$ and from the symmetry relation $r(\bsk,\bsl)=r(\bsl,\bsk)$ we obtain

\begin{eqnarray*}
\LL_{2,N}^2(\cP_{N,s}) &= & \frac{1}{3^s} - \frac{2}{N}
\sum_{n=0}^{N-1} \sum_{\bsk \in \NN_0^s} r(\bsk,\bszero)
\wal_{\bsk}(\bsx_n) \\ && + \frac{1}{N^2} \sum_{n,m=0}^{N-1}
\sum_{\bsk,\bsl \in \NN_0^s} r(\bsk,\bsl) \wal_{\bsk}(\bsx_n)
\wal_{\bsl}(\bsx_m) \\ & = & r(\bszero,\bszero)-2 r(\bszero,\bszero) +r(\bszero,\bszero)\\
&  & - \frac{1}{N^2}\sum_{n,m=0}^{N-1} \sum_{\bsk \in \NN^s} r(\bsk,\bszero) \wal_{\bsk}(\bsx_n) - \frac{1}{N^2} \sum_{n,m=0}^{N-1} \sum_{\bsl \in \NN^s} r(\bszero,\bsl) \wal_{\bsl}(\bsx_m)\\
& & + \frac{1}{N^2} \sum_{n,m=0}^{N-1}\sum_{\bsk,\bsl \in \NN_0^s \atop(\bsk,\bsl)\not=(\bszero,\bszero)} r(\bsk,\bsl) \wal_{\bsk}(\bsx_n)
\wal_{\bsl}(\bsx_m)\\
& = & \frac{1}{N^2} \sum_{n,m=0}^{N-1}\sum_{\bsk,\bsl \in \NN_0^s \setminus \{\bszero\}} r(\bsk,\bsl) \wal_{\bsk}(\bsx_n)
\wal_{\bsl}(\bsx_m)\\
& = & \sum_{\bsk,\bsl \in \NN_0^s \setminus\{\bszero\}} r(\bsk,\bsl) \frac{1}{N} \sum_{n=0}^{N-1} \wal_{\bsk}(\bsx_n)  \frac{1}{N} \sum_{m=0}^{N-1} \wal_{\bsl}(\bsx_m),
\end{eqnarray*}
which proves the first part. The second part follows immediately from the first part and the character property \eqref{charprop} of digital nets. The third part follows in the same manner as the second part using the additional equality $\wal_{\bsk}(\bsx \oplus \bssigma) = \wal_{\bsk}(\bsx) \wal_{\bsk}(\bssigma)$.
\end{proof}

\section{The proof of Theorem~\ref{thm1}}

We give the proof of our main result. Throughout this proof we assume that $\alpha \ge 3$ unless stated otherwise. We consider the construction of digital sequences $\cS_{\alpha s}$ based on \eqref{def_sob_mat} in dimension $\alpha s$ and apply the digit interlacing function $\mathscr{D}_\alpha^s(\cS_{\alpha s})$ of order $\alpha$. The sequence $\cS_s:=\mathscr{D}_\alpha^s(\cS_{\alpha s}):=(\boldsymbol{x}_0, \boldsymbol{x}_1,\ldots)$ in $[0,1)^s$ is an order $\alpha$ digital $(t,s)$-sequence with $t = \alpha \sum_{j=1}^s (e_j-1) + s {\alpha \choose 2}$. Using \eqref{eq_t_tprime}, $\mathscr{D}_\alpha^s(\cS_{\alpha s})$ is also an order $\alpha'$ digital $(t',s)$-sequence with $t' = \lceil t \alpha'/\alpha \rceil \le t$ for all $1 \le \alpha' \le \alpha$. Thus it is also an order $\alpha'$ digital $(t,s)$-sequence for all $1 \le \alpha' \le \alpha$, see Subsection~\ref{sec_gen_princ} for more details. Note that we have $t \ge s {\alpha \choose 2} \ge {3 \choose 2} = 3$.

Let $C_1, \ldots, C_s$ denote the generating matrices of the digital sequence $\cS_{s}$. Let $C_{j, \mathbb{N} \times m}$ denote the first $m$ columns of $C_j$. As explained in Subsection~\ref{sec_gen_princ}, only the first $\alpha m$ rows of $C_{j, \mathbb{N} \times m}$ can be nonzero and hence $C_j$ is of the form
$$
C_{j} = \left( \begin{array}{ccc}
           & \vline &  \\
  C_{j,\alpha m \times m} & \vline & D_{j,\alpha m \times \NN} \\
           & \vline &   \\ \hline
    & \vline & \\
 0_{\NN \times m}  & \vline &  F_{j,\NN \times \NN} \\
   &   \vline       &
\end{array} \right) \in \FF_2^{\NN \times \NN},
$$  where $0_{\NN \times m}$ denotes the $\NN \times m$ zero matrix. Note that the entries of each column of the   matrix $F_{j,\NN \times \NN}$ become eventually zero.

We use the first part of Lemma~\ref{lem_l2formula} to obtain
\begin{align}\label{fo1l2}
\LL^2_{2,N}(\cS_s)  = & \sum_{\bsk,\bsl \in \NN_0^s \setminus\{\bszero\}}
r(\bsk,\bsl) \frac{1}{N}\sum_{n=0}^{N-1} \wal_{\bsk}(\bsx_n)
\frac{1}{N} \sum_{m=0}^{N-1} \wal_{\bsl}(\bsx_{m}).
\end{align}
Let $N = 2^{m_1} + 2^{m_2} + \cdots + 2^{m_r}$ with $m_1 > m_2 > \cdots > m_r \ge 0$ (hence $r=S(N)$). We consider the point sets
\begin{equation*}
\cP_i:=\{\boldsymbol{x}_{2^{m_1}+\cdots + 2^{m_{i-1}}},\ldots, \boldsymbol{x}_{-1+2^{m_1} + \cdots + 2^{m_i}}\},
\end{equation*}
for $i=1,\ldots,r$, where for $i=1$ we define $2^{m_1} + \cdots + 2^{m_{i-1}} = 0$. Any $n \in \{2^{m_1}+\cdots + 2^{m_{i-1}}, \ldots ,-1+2^{m_1} + \cdots + 2^{m_i}\}$ can be written in the form $$n=2 ^{m_1}+\cdots+2^{m_{i-1}}+a=2^{m_{i-1}} \ell +a$$ with $a \in \{0,1,\ldots,2^{m_i}-1\}$ and $\ell=1+2^{m_{i}-m_{i-1}}+\cdots+2^{m_1-m_{i-1}}$ if $i > 1$ and $\ell =0$ for $i=1$. Hence the dyadic digit vector of $n$ is given by $$\vec{n}=(a_0,a_1,\ldots,a_{m_i-1},l_0,l_1,l_2,\ldots)^{\top}=:{\vec{a} \choose \vec{\ell}},$$ where $a_0,\ldots,a_{m_i-1}$ are the dyadic digits of $a$ and $l_0,l_1,l_2,\ldots$ are the dyadic digits of $\ell$. With this notation we have
$$C_j \vec{n}=\left( \begin{array}{c} C_{j,\alpha m_i \times m_i}  \vec{a} \\ 0 \\ 0 \\ \vdots \end{array} \right)  + \left( \begin{array}{c}
             \\
  D_{j,\alpha m \times \NN}  \\
          \\ \hline  \\
 F_{j,\NN \times \NN}  \\
   \end{array} \right) \vec{\ell}.$$ For the point set $\cP_i$ under consideration, the vector
\begin{equation}\label{dig_shift}
\vec{\sigma}_{i,j}:=\left( \begin{array}{c}
             \\
  D_{j,\alpha m \times \NN}  \\
          \\ \hline  \\
 F_{j,\NN \times \NN}  \\
   \end{array} \right) \vec{\ell}
\end{equation}
is constant and its components become eventually zero (i.e., only a finite number of components is nonzero). Furthermore, $C_{j,\alpha m_i \times m_i}  \vec{a}$ for $a=0,1,\ldots,2^{m_i}-1$ and $j=1,\ldots,s$ generate an order $\alpha$ digital $(t,m_i,s)$-net over $\FF_2$ (which is also an order $\alpha'$ digital $(t,m_i,s)$-net over $\FF_2$ for $1 \le \alpha' \le \alpha$).

This means that the point set $\cP_i$ is a digitally shifted order $\alpha$ digital $(t,m_i,s)$-net over $\FF_2$  and the generating matrices
\begin{equation}\label{genmatpi}
C_{1, \alpha m_i \times m_i},\ldots, C_{s, \alpha m_i\times m_i}
\end{equation}
 of this digital net are the left upper $\alpha m_i \times m_i$ submatrices of the generating matrices $C_1,\ldots, C_s$ of the digital sequence. We denote the digital shift, which is given by \eqref{dig_shift}, by $\bssigma_i$. Note that all the coordinates of the digital shift are dyadic rationals since the components of $\vec{\sigma}_{i,j}$ become eventually zero.

Let $\mathcal{D}_i$ denote the dual net corresponding to the digital net with generating matrices \eqref{genmatpi}, i.e.,
\begin{equation*}
\mathcal{D}_i = \{\bsk=(k_1,\ldots,k_s) \in \mathbb{N}_0^s\,:\, C_{1, \alpha m_i \times m_i}^\top \vec{k}_1 + \cdots + C_{s, \alpha m_i \times m_i}^\top \vec{k}_s = \vec{0}\},
\end{equation*}
where for $k \in \NN_0$ with base 2 expansion $k = \kappa_0 + \kappa_1 2 + \kappa_2 2^2+\cdots$ we set $\vec{k} = (\kappa_0, \kappa_1,\ldots, \kappa_{\alpha m_i-1})^\top$. Set $\Dcal_i^{\ast}=\Dcal_i \setminus\{\bszero\}$. 

We now obtain a bound on the $\LL_2$ discrepancy using the dual nets $\mathcal{D}_i$.
\begin{lemma}\label{lem_proof_1}
Let $N = 2^{m_1} + 2^{m_2} + \cdots + 2^{m_r}$ where $m_1 > m_2 > \cdots > m_r \ge 0$. Using the notation from above, let
\begin{equation}\label{def_Jii}
\Jcal_{i,i'} = \{(\bsk,\bsl) \in \Dcal^\ast_i \times \Dcal^\ast_{i'} \ : \ r(\bsk,\bsl) \neq 0\}
\end{equation}
and
\begin{equation}\label{def_Jii_z}
\mathcal{J}_{i,i'}(z) = \{(\boldsymbol{k},\boldsymbol{l}) \in \mathcal{J}_{i,i'} \ : \ \mu(\boldsymbol{k}) + \mu(\boldsymbol{l}) = z\}.
\end{equation}
Then we have
\begin{equation}\label{eq_bound_Jiiz}
\LL_{2,N}^2(\cS_s) \ll_s \sum_{i,i'=1}^r \frac{2^{m_i}}{N} \frac{2^{m_{i'}}}{N} \sum_{z=m_i + m_{i'}-2t+2}^\infty \frac{|\mathcal{J}_{i,i'}(z)|}{2^z}.
\end{equation}
\end{lemma}

\begin{proof}
By the character property \eqref{charprop} we have
\begin{align*}
\frac{1}{2^{m_i}} \sum_{n=2^{m_1}+\cdots + 2^{m_{i-1}}}^{-1+2^{m_1}+\cdots + 2^{m_i}} \wal_{\bsk}(\bsx_n) = & \left\{\begin{array}{ll} \wal_{\bsk}(\bssigma_i) & \mbox{if } \bsk \in \mathcal{D}_i, \\ 0 & \mbox{if } \bsk \notin \mathcal{D}_i, \end{array} \right.
\end{align*}
where again for $i=1$ we set $2^{m_1}+\cdots + 2^{m_{i-1}} = 0$, and hence
\begin{align*}
\frac{1}{N}\sum_{n=0}^{N-1} \wal_{\bsk}(\bsx_n) = & \sum_{i=1}^r \frac{2^{m_i}}{N} \frac{1}{2^{m_i}} \sum_{n=2^{m_1}+\cdots + 2^{m_{i-1}}}^{-1+2^{m_1}+\cdots + 2^{m_i}} \wal_{\bsk}(\bsx_n)\\
= & \sum_{i=1 \atop \bsk \in \mathcal{D}_i}^r \frac{2^{m_i}}{N} \wal_{\bsk}(\bssigma_i).
\end{align*}
Inserting this into \eqref{fo1l2} and interchanging the order of summation we obtain
\begin{align}\label{def_sum_bound}
\LL^2_{2,N}(\cS_s)  = &  \sum_{i, i' =1}^r \frac{2^{m_i}}{N} \frac{2^{m_{i'}}}{N} \sum_{(\boldsymbol{k},\boldsymbol{l}) \in \mathcal{D}^{\ast}_i  \times \mathcal{D}^{\ast}_{i'}} r(\boldsymbol{k}, \boldsymbol{l}) \wal_{\boldsymbol{k}}(\bssigma_i) \wal_{\bsl}(\bssigma_{i'}) \nonumber \\ \le & \sum_{i, i' =1}^r \frac{2^{m_i}}{N} \frac{2^{m_{i'}}}{N} \sum_{(\boldsymbol{k},\boldsymbol{l}) \in \Jcal_{i,i'}} |r(\boldsymbol{k}, \boldsymbol{l})|,
\end{align}
since $|\wal_{\bsk}(\bsx)|=1$ for any $\bsx$.

For a vector $\bsk=(k_1,\ldots,k_s) \in \NN_0^s$ we put $\mu(\bsk)=\sum_{j=0}^s \mu(k_j)$, where, as already defined earlier, the function $\mu: \mathbb{N}_0 \to \mathbb{N}_0$ is defined by $\mu(0)=0$ and for $k = \kappa_0 + \kappa_1 2 + \cdots + \kappa_{a-2} 2^{a-2} + 2^{a-1}$ with $\kappa_j \in\{0,1\}$ by $\mu(k)=a$.

According to the definition of $r(\bsk,\bsl)$ in Lemma~\ref{lem_r} for $(\bsk,\bsl) \in \Jcal_{i,i'}$ we have $$|r(\bsk,\bsl)| \le \frac{1}{3^s  2^{\mu(\bsk) + \mu(\bsl)}}.$$ Thus we obtain from \eqref{def_sum_bound}
\begin{equation}\label{def_sum_bound2}
\LL_{2,N}^2(\cS_s) \ll_s  \sum_{i,i'=1}^r \frac{2^{m_i}}{N} \frac{2^{m_{i'}}}{N} \sum_{(\bsk,\bsl) \in \Jcal_{i,i'}} \frac{1}{2^{\mu(\bsk) + \mu(\bsl)}}.
\end{equation}

Now we re-order the sum over all $(\boldsymbol{k},\boldsymbol{l}) \in \Jcal_{i,i'}$ according to the value of $\mu(\boldsymbol{k}) + \mu(\boldsymbol{l})$. 

Assume that $\bsk=(k_1,\ldots,k_s) \in \Dcal_i^\ast$. Let $k_j=\kappa_{j,0}+\kappa_{j,1} 2 +\cdots +\kappa_{j,a_j-2} 2^{a_j-2}+2^{a_j -1}$ with $a_j=\mu(k_j)$ for $j=1,\ldots,s$. Let further $\vec{c}_{j,u}$ denote the $u$th row vector of the matrix $C_{j,\alpha m_i \times m_i}$. Then $$C_{1, \alpha m_i \times m_i}^\top \vec{k}_1 + \cdots + C_{s, \alpha m_i \times m_i}^\top \vec{k}_s = \vec{0}$$ is equivalent to $$\sum_{j=1}^s \left(\sum_{u=0}^{a_j-2} \vec{c}_{j,u+1}^{\ \top} \kappa_{j,u} + \vec{c}_{j,a_j-1}^{\ \top}\right)=\vec{0}.$$ Hence it follows from the linear independence property for the row vectors of generating matrices of digital nets in Definition~\ref{def_net} that $$\mu(\bsk)=a_1+\cdots+a_s > m_i -t.$$ In the same way $\bsl \in \Dcal_{i'}^\ast$ implies that $\mu(\boldsymbol{l}) > m_{i'}-t$. Hence $(\bsk,\bsl)\in \Dcal^\ast_i \times \Dcal^\ast_{i'}$ implies $\mu(\bsk) + \mu(\bsl) \ge m_i + m_{i'} -2t+2$.

Thus for the innermost sum in \eqref{def_sum_bound2} we have $$\sum_{(\bsk,\bsl) \in \mathcal{J}_{i,i'}} \frac{1}{2^{\mu(\bsk)+\mu(\bsl)}} = \sum_{z= m_i + m_{i'} - 2t + 2}^\infty \frac{|\Jcal_{i,i'}(z)|}{2^z}.$$ By substituting this result into \eqref{def_sum_bound2} the result follows.
\end{proof}

To obtain a bound on the right-hand side of \eqref{eq_bound_Jiiz}, we first obtain a bound on the number of elements in the set $\Jcal_{i,i'}(z)$. We do this in the next six lemmas.

\begin{lemma}\label{lem_bound_factors}
Using the notation from above, we have
\begin{equation}\label{bdJii}
|\mathcal{J}_{i,i'}(z)| = \sum_{z_1=m_i-t+1}^{z-m_{i'} +t-1} |\{(\bsk,\bsl) \in \mathcal{J}_{i,i'} \ : \ \mu(\bsk)=z_1 \ \mbox{ and } \ \mu(\bsl)=z-z_1\}|.
\end{equation}
\end{lemma}

\begin{proof}
We have 
\begin{equation*}
|\mathcal{J}_{i,i'}(z)| = \sum_{z_1=0}^{z} |\{(\bsk,\bsl) \in \mathcal{J}_{i,i'} \ : \ \mu(\bsk)=z_1 \ \mbox{ and } \ \mu(\bsl)=z-z_1\}|.
\end{equation*}
Now $(\boldsymbol{k}, \boldsymbol{l}) \in \mathcal{J}_{i,i'}$ implies $\bsk \in \mathcal{D}_i^\ast$ and $\bsl \in \mathcal{D}_{i'}^\ast$. We already showed in the proof of Lemma~\ref{lem_proof_1} that
 $\bsk \in \mathcal{D}_i^{\ast}$ implies that $\mu(\boldsymbol{k}) > m_{i}-t$ and $\bsl \in \mathcal{D}_{i'}^{\ast}$ implies that $\mu(\boldsymbol{l}) > m_{i'}-t$. Thus we only need to consider the case where $z_1 > m_i-t$ and $z-z_1 > m_{i'} -t$ and hence the result follows.
\end{proof}

\begin{lemma}\label{lem5neu}
Using the notation from above, we have
\begin{eqnarray*}
\lefteqn{|\mathcal{J}_{i,i'}(z)|}\\ & \le & \sum_{z_1=m_i-t+1}^{z-m_{i'} +t-1} \min\left\{|\{\boldsymbol{k} \in \mathcal{D}_i^{\ast}\,:\, \mu(\boldsymbol{k}) = z_1\}| \max_{\satop{\boldsymbol{k} \in \mathcal{D}_i^{\ast}}{\mu(\boldsymbol{k}) = z_1}}|R_{i,i'}^{(1)}(\boldsymbol{k}, z-z_1)|,\right.\\
&& \hspace{3cm}\left. |\{\boldsymbol{l} \in \mathcal{D}_{i'}^{\ast}\,:\, \mu(\boldsymbol{k}) = z-z_1\}| \max_{\satop{\boldsymbol{l} \in \mathcal{D}_{i'}^{\ast}}{\mu(\boldsymbol{l}) = z-z_1}}|R_{i,i'}^{(2)}(\boldsymbol{l}, z_1)|\right\},
\end{eqnarray*}
where
\begin{eqnarray*}
R_{i,i'}^{(1)}(\boldsymbol{k}, z-z_1) & = & \{\boldsymbol{l} \in \mathcal{D}^\ast_{i'}\,:\, (\boldsymbol{k},\boldsymbol{l}) \in \mathcal{J}_{i,i'}(z) \mbox{ and } \mu(\boldsymbol{l})=z-z_1\}\\
R_{i,i'}^{(2)}(\boldsymbol{l}, z_1) & = & \{\boldsymbol{k} \in \mathcal{D}^\ast_{i}\,:\, (\boldsymbol{k},\boldsymbol{l}) \in \mathcal{J}_{i,i'}(z) \mbox{ and } \mu(\boldsymbol{k})=z_1\}.
\end{eqnarray*} 

\end{lemma}

\begin{proof}
Each summand in \eqref{bdJii} can be estimated on the one hand by 
\begin{eqnarray*}
\lefteqn{|\{(\bsk,\bsl) \in \mathcal{J}_{i,i'} \ : \ \mu(\bsk)=z_1 \ \mbox{ and } \ \mu(\bsl)=z-z_1\}|}\\
& \le &  |\{\boldsymbol{k} \in \mathcal{D}_i^{\ast}\,:\, \mu(\boldsymbol{k}) = z_1\}| \max_{\satop{\boldsymbol{k} \in \mathcal{D}_i^{\ast}}{\mu(\boldsymbol{k}) = z_1}}|R_{i,i'}^{(1)}(\boldsymbol{k}, z-z_1)|,
\end{eqnarray*}
and on the other hand by
\begin{eqnarray*}
\lefteqn{|\{(\bsk,\bsl) \in \mathcal{J}_{i,i'} \ : \ \mu(\bsk)=z_1 \ \mbox{ and } \ \mu(\bsl)=z-z_1\}|}\\
& \le &  |\{\boldsymbol{l} \in \mathcal{D}_{i'}^{\ast}\,:\, \mu(\boldsymbol{k}) = z-z_1\}| \max_{\satop{\boldsymbol{l} \in \mathcal{D}_{i'}^{\ast}}{\mu(\boldsymbol{l}) = z-z_1}}|R_{i,i'}^{(2)}(\boldsymbol{l}, z_1)|.
\end{eqnarray*}
Hence the result follows.
\end{proof}

To prove the following results we introduce some notation. Let $k_j, l_j \in \mathbb{N}_0$. In the following we simultaneously use two different notations for the binary expansion of $k_j$ and $l_j$.  First let $$k_j = 2^{a_{j,1}-1} + \cdots + 2^{a_{j,\widetilde{v}_j}-1}$$ with $a_{j, 1} > \cdots > a_{j,\widetilde{v}_j} > 0$ and $$l_j = 2^{b_{j,1}-1} + \cdots + 2^{b_{j, \widetilde{w}_j}-1}$$ with $b_{j,1} > \cdots > b_{j, \widetilde{w}_j} > 0$. Thus $\widetilde{v}_j$ denotes the number of nonzero digits of $k_j$ and $\widetilde{w}_j$ denotes the number of nonzero digits of $l_j$. For $k_j = 0$ we use the convention that $\widetilde{v}_j = 0$ and $a_{j,1} = 0$. Further we set $a_{j,\widetilde{v}_j+i} = b_{j,\widetilde{w}_j+i} = 0$ for $i > 0$.

We also use the notation $$k_j =  k_{j,0} + k_{j,1}2 + \cdots + k_{j,a_{j,1} -1} 2^{a_{j,1} -1}$$ with binary digits $k_{j,i} \in \{0,1\}$. Thus  $$k_{j,i} = \left\{\begin{array}{rl} 1 & \mbox{if } i = a_{j,v} \mbox{ for some } 1 \le v \le \widetilde{v}_j, \\ 0 & \mbox{otherwise}. \end{array} \right.$$ Analogously we write $$l_j =  l_{j,0} + l_{j,1}2 + \cdots + l_{j,b_{j,1} - 1} 2^{b_{j,1} -1}$$ with binary digits $l_{j,i} \in \{0,1\}$. Thus $$l_{j,i} = \left\{\begin{array}{rl} 1 & \mbox{if } i = b_{j,w} \mbox{ for some } 1 \le w \le \widetilde{w}_j, \\ 0 & \mbox{otherwise}. \end{array} \right.$$

We now study the factors appearing in the bound in Lemma~\ref{lem_bound_factors}  separately in two steps.

\begin{lemma}\label{lem_bound_D}
For $z_1 \ge m_i-t+1$ we have
\begin{equation*}
|\{\boldsymbol{k} \in \mathcal{D}_i^{\ast} \,:\, \mu(\boldsymbol{k}) = z_1\}| \ll_s  {z_1 + s-1 \choose s-1} 2^{z_1 -m_i +t-1}
\end{equation*}
and for $z-z_1 \ge m_{i'}-t+1$ we have
\begin{equation*}
|\{\boldsymbol{l} \in \mathcal{D}_{i'}^{\ast} \,:\, \mu(\boldsymbol{l}) = z-z_1\}| \ll_s  {z-z_1 + s-1 \choose s-1} 2^{z-z_1 -m_{i'} +t-1}.
\end{equation*}
\end{lemma}

\begin{proof}
It suffices to show the first estimate, the second estimate is a direct consequence of the first bound. The number of $\boldsymbol{k} = (k_1,\ldots, k_s) \in \mathcal{D}_i^{\ast}$ with $\mu(\boldsymbol{k}) = z_1$ has been studied in \cite{DP05b}. Assume first that $k_j > 0$ for $1 \le j \le s$. The case where one or more of the $k_j$'s are zero follows by the same arguments. Let $\Sigma(v_{1}, v_{2}, \ldots, v_{s})$ denote the number of such $\boldsymbol{k} = (k_1,\ldots, k_s) \in \mathcal{D}_i^{\ast}$ with $\mu(k_j) = a_{j,1} = v_j$. Then $\boldsymbol{k} \in \mathcal{D}_i^{\ast}$ implies that
\begin{eqnarray}\label{lgs1.1}
\vec{c}_{1,1}^{\ \top} k_{1,0}+\cdots +\vec{c}_{1,v_{1}-1}^{\ \top}
k_{1,v_{1}-2}+  \vec{c}_{1,v_{1}}^{\ \top} + &&\nonumber \\
\vec{c}_{2,1}^{\ \top} k_{2,0}+\cdots + \vec{c}_{2,v_{2}-1}^{\ \top}
k_{2,v_{2} -2}+\vec{c}_{2,v_{2}}^{\ \top} + &&\nonumber\\
\vdots && \\
\vec{c}_{s,1}^{\ \top} k_{s,0}+\cdots +\vec{c}_{s,v_{s}-1}^{\ \top}
k_{s,v_{s}-2}+ \vec{c}_{s,v_{s}}^{\ \top} \hspace{0.3cm}& = & \vec{0},\nonumber
\end{eqnarray}
where $\vec{c}_{j,u} \in \FF_2^{m_i}$ denotes the $u$th row vector of the matrix $C_{j,\alpha m_i \times m_i}$. Since by the (order $1$) digital $(t,m_i,s)$-net property the vectors
$$\vec{c}_{1,1},\ldots ,\vec{c}_{1,v_1},\ldots ,\vec{c}_{s,1},\ldots
,\vec{c}_{s,v_s}$$ are linearly independent as long as
$v_1+\cdots +v_s \le m_i-t,$ we must have
\begin{eqnarray}\label{bed1}
v_1+\cdots +v_s \ge m_i-t+1.
\end{eqnarray}

Let now $A$ denote the $m_i \times ((v_1-1) + \cdots + (v_s-1))$ matrix with
column vectors $\vec{c}_{1,1}^{\ \top},\ldots ,\vec{c}_{1,v_1-1}^{\ \top},\ldots ,\vec{c}_{s,1}^{\ \top},\ldots
,\vec{c}_{s,v_s-1}^{\ \top}$, i.e., $$A:=(\vec{c}_{1,1}^{\ \top},\ldots ,\vec{c}_{1,v_1-1}^{\ \top},\ldots ,\vec{c}_{s,1}^{\ \top},\ldots
,\vec{c}_{s,v_s-1}^{\ \top}).$$ Further let $$\vec{f}:=\vec{c}_{1,v_1}^{\ \top}+
\cdots + \vec{c}_{s,v_s}^{\ \top}$$ and
$$\vec{k}:=\underbrace{(k_{1,0},\ldots ,k_{1,v_1-2},\ldots ,k_{s,0},\ldots
,k_{s,v_s-2})^{\top}}_{\mbox{\scriptsize {\rm  length}}\ (v_1-1)+\cdots+ (v_s-1)}.$$
Then the linear system of equations \eqref{lgs1.1} can be written as
\begin{eqnarray}\label{lgs2}
A\vec{k} = \vec{f}
\end{eqnarray}
and hence
\begin{eqnarray*}
\Sigma(v_1,\ldots,v_s) = \sum_{\vec{k}\in \mathbb{F}_2^{(v_1-1) +\cdots + (v_s-1)}\atop A \vec{k}=\vec{f}}1= |\{\vec{k}\in \mathbb{F}_2^{(v_1-1) +\cdots
+ (v_s-1)}: A \vec{k}=\vec{f}\}|.
\end{eqnarray*}
By the definition of the matrix $A$ and since $C_{1,\alpha m_i \times m_i},\ldots ,C_{s,\alpha m_i \times m_i}$ are the generating matrices of an (order $1$)  digital $(t,m_i,s)$-net over $\mathbb{F}_2$ we have
$${\rm rank}(A)= \left\{
\begin{array}{ll}
(v_1-1) +\cdots + (v_s-1) & \mbox{if } (v_1-1)+\cdots + (v_s-1) \le m_i - t,\\
\ge m_i-t & \mbox{otherwise}.
\end{array}
\right.$$
Let $L$ denote the linear space of solutions of the homogeneous system
$A\vec{k}=\vec{0}$ and let ${\rm dim}(L)$ denote the dimension of
$L$. Then it follows that $${\rm dim}(L) = \left\{
\begin{array}{ll}
0 & \mbox{if } v_1+\cdots +v_s\le m_i - t + s,\\
\le v_1+\cdots +v_s-m_i + t - s & \mbox{otherwise}.
\end{array}
\right.$$
Hence if $v_1+\cdots +v_s \le m_i - t + s$ we find that the system \eqref{lgs2}
has at most 1 solution and if $v_1+\cdots +v_s > m_i - t + s$ the system
\eqref{lgs2} has at most $2^{v_1+\cdots+v_s-m_i + t - s}$ solutions, i.e.,
$$\Sigma(v_1,\ldots,v_s) \le \left\{
\begin{array}{ll}
1 & \mbox{ if } v_1+\cdots +v_s \le m_i - t + s,\\
2^{v_1+\cdots+v_s-m_i +t-s} & \mbox{ if } v_1+\cdots +v_s > m_i -t+s.
\end{array}
\right.$$
Recall that $v_1+\cdots+v_s=\mu(\bsk)$.

In the following let ${n \choose k}$ denote the binomial coefficient, where we set ${n \choose k} = 0$ if $k > n$. Thus we have
\begin{eqnarray*}
\lefteqn{|\{\boldsymbol{k} \in \mathcal{D}_i^{\ast} \,:\, k_j > 0 \mbox{ for } j=1,\ldots,s \mbox{ and }  \mu(\boldsymbol{k}) = z_1\}|} \\ & = & \left\{\begin{array}{ll} {z_1 + s-1 \choose s-1} & \mbox{if } z_1 \le m_i -t+s, \\ {z_1 + s-1 \choose s-1} 2^{z_1 -m_i +t-s}  & \mbox{if } z_1 > m_i -t+s.  \end{array} \right.
\end{eqnarray*}
In general, for $\emptyset \neq u \subseteq \{1,\ldots,s\}$ we have
\begin{eqnarray*}
\lefteqn{|\{\boldsymbol{k} \in \mathcal{D}_i^{\ast}\,:\, k_j > 0 \mbox{ for } j \in u, k_j=0 \mbox{ otherwise, and } \mu(\boldsymbol{k}) = z_1\}|} \\  & = & \left\{\begin{array}{ll} {z_1  + |u| - 1 \choose |u|-1} & \mbox{if } z_1 \le m_i -t+|u|, \\ {z_1 + |u| - 1 \choose |u| - 1} 2^{z_1 - m_i +  t - |u|}  & \mbox{if } z_1 > m_i -t+|u|.  \end{array} \right.
\end{eqnarray*}
Thus, in general, for $z_1 \ge m_i-t+1$ we have
\begin{equation*}
|\{\boldsymbol{k} \in \mathcal{D}_i^{\ast} \,:\, \mu(\boldsymbol{k}) = z_1\}| \ll_s  {z_1 + s-1 \choose s-1} 2^{z_1 -m_i +t-1}.
\end{equation*}
\end{proof}

\begin{lemma}\label{lem_bound_R}
Let $R_{i,i'}^{(1)}(\boldsymbol{k}, z-z_1)$ and $R_{i,i'}^{(2)}(\bsl,z_1)$ be defined as in Lemma~\ref{lem5neu}.  Then for $\bsk \in \mathcal{D}_i^\ast$ we have
\begin{eqnarray*}
|R_{i,i'}^{(1)}(\boldsymbol{k},z-z_1)| \le   {2(z-z_1)-2m_{i'} +t+ s \choose s}{3(z-z_1) - 3m_{i'} +t + s \choose s},
\end{eqnarray*}
and for $\bsl \in \mathcal{D}_{i'}^\ast$ we have
\begin{eqnarray*}
|R_{i,i'}^{(2)}(\boldsymbol{l},z_1)| \le   {2z_1-2m_{i} +t+ s \choose s}{3z_1 - 3m_{i} +t + s \choose s}.
\end{eqnarray*}
\end{lemma}

\begin{proof}
Again it suffices to show the first estimate, the second estimate follows by the same arguments. For the proof of this result we first need to analyze for which $(\bsk, \bsl) \in \mathcal{D}_i \times \mathcal{D}_{i'}$ the factors $r(\bsk,\bsl) \neq 0$. To do so we need to consider a number of cases.

Recall that $r(\boldsymbol{k}, \boldsymbol{l}) = \prod_{j=1}^s r(k_j, l_j)$. For $r(k_j,l_j) \neq 0$ it follows that in some sense $k_j$ and $l_j$ cannot be too different. Let us elaborate in more detail: Assume that $r(k_j,l_j) \neq 0$. Now Lemma~\ref{lem_r} implies that in order for $r(k_j,l_j)$ not to be $0$ we must have $0 \le |\widetilde{v}_j - \widetilde{w}_j| \le 2$. Further we must have:
\begin{equation*}
\begin{array}{llll}
(i)   & |\widetilde{v}_j - \widetilde{w}_j| = 0 & \Rightarrow & a_{j,2} = b_{j,2}, \ldots, a_{j,\widetilde{v}_j} = b_{j,\widetilde{v}_j}, \\
(ii)  & |\widetilde{v}_j - \widetilde{w}_j| = 1 & \Rightarrow & k_j = 0 \mbox{ or } l_j = 0, \\
(iii) & |\widetilde{v}_j - \widetilde{w}_j| = 2 & \Rightarrow & \mbox{if } \widetilde{v}_j = \widetilde{w}_j + 2 \mbox{ then } a_{j,3} = b_{j,1}, \ldots, a_{j,\widetilde{v}_j} = b_{j,\widetilde{w}_j} \\ & & & \mbox{if } \widetilde{w}_j = \widetilde{v}_j + 2 \mbox{ then } b_{j,3} = a_{j,1}, \ldots, b_{j,\widetilde{w}_j} = a_{j,\widetilde{v}_j}.
\end{array}
\end{equation*}
If $|\widetilde{v}_j - \widetilde{w}_j| > 2$ we always have $r(k_j,l_j) = 0$.

For given $(\boldsymbol{k}, \boldsymbol{l}) \in \mathcal{J}_{i,i'}(z)$ we define the following sets for $-2 \le \tau \le 2$:
\begin{align*}
\alpha_\tau & = \{j \in \{1,\ldots,s\}\,:\, \widetilde{v}_j = \widetilde{w}_j+ \tau\}.
\end{align*}
Note that $\alpha_{\tau} \cap \alpha_{\tau'} = \emptyset$ for $\tau \neq \tau'$ and $\bigcup_{\tau=-2}^2 \alpha_\tau = \{1,\ldots,s\}$ by Lemma~\ref{lem_r}. Then we have
\begin{enumerate}
\item For $j \in \alpha_2$ we have $l_{j,i} = k_{j,i}$ for $0 \le i < a_{j,2} -1$;
\item For $j \in \alpha_1$ we have $l_j=0$ and $k_j = 2^{a_{j,1} -1}$;
\item For $j \in \alpha_0$ we have $k_{j,i} = l_{j,i}$ for $0 \le i < \min\{a_{j,1}, b_{j,1} \}-1$;
\item For $j \in \alpha_{-1}$ we have $l_j = 2^{b_{j,1} - 1}$ and $k_j=0$;
\item For $j \in \alpha_{-2}$ we have $l_{j,i}=k_{j,i}$ for $0 \le i < b_{j,2} -1$.
\end{enumerate}
Thus in all cases we have $k_{j,i} = l_{j,i}$ for $0 \le i < \min\{a_{j,2} -1, b_{j,2}-1\}$. We set now
\begin{equation*}
h_{j,i} = k_{j,i} = l_{j,i} \quad \mbox{for all } 1 \le j \le s \mbox{ and } 0 \le i < \min\{a_{j,2} -1, b_{j,2} -1\},
\end{equation*}
and for $u_j = \min\{a_{j,2} -1, b_{j,2} -1\}$ we set
\begin{equation*}
h_j = h_{j,0} + h_{j,1} 2 + \cdots + h_{j,u_j-1} 2^{u_j-1} \quad \mbox{for } 1 \le j \le s
\end{equation*}
if $u_j > 0$ and $h_j=0$ otherwise.
Thus we only need to consider the cases where
\begin{align*}
k_j & = h_j + \lfloor 2^{a_{j,2} -1} \rfloor + \lfloor 2^{a_{j,1} -1} \rfloor, \\
l_j & = h_j + \lfloor 2^{b_{j,2} -1} \rfloor + \lfloor 2^{b_{j,1} -1} \rfloor
\end{align*}
for $1 \le j \le s$.

We now prove a bound on $|R_{i,i'}^{(1)}(\boldsymbol{k}, z-z_1)|$. Let now $\vec{c}_{j,u}$ denote the $u$th row of the matrix $C_{j, \alpha m_{i'} \times m_{i'}}$.

Let $\boldsymbol{k} \in \mathcal{D}^\ast_i$ be fixed and $\mu(k_j) = a_{j,1}$ for $1 \le j \le s$. We have $|\widetilde{v}_j - \widetilde{w}_j| \le 2$ and $\boldsymbol{l} \in \mathcal{D}_{i'}^{\ast}$ implies that
\begin{eqnarray*}%\label{lgs3}
\vec{c}_{1,1}^{\ \top} h_{1,0}+\cdots +  \vec{c}_{1,b_{1,2} -1}^{\ \top}
h_{1,b_{1,2} -2}+ \vec{c}_{1,b_{1,2} }^{\ \top} + \vec{c}_{1,b_{1,1}}^{\ \top}+&&\nonumber\\
\vec{c}_{2,1}^{\ \top} h_{2,0}+\cdots +\vec{c}_{2,b_{2,2} -1}^{\ \top}
h_{2, b_{2,2} -2}+ \vec{c}_{2,b_{2,2} }^{\ \top} + \vec{c}_{2,b_{2,1}}^{\ \top}+&&\nonumber\\
\vdots && \\
\vec{c}_{s,1}^{\ \top} h_{s,0}+\cdots +\vec{c}_{s,b_{s,2} -1}^{\ \top}
h_{s,b_{s,2} -2} + \vec{c}_{s,b_{s,2} }^{\ \top} + \vec{c}_{s,b_{s,1} }^{\ \top} \hspace{0.3cm}& = & \vec{0}. \nonumber
\end{eqnarray*}
If $b_{j,1}$ or $b_{j,2}$ is zero, we set $\vec{c}_{j,0}^{\ \top} = \vec{0}$. Note that we consider $\boldsymbol{k}$ to be fixed, thus the $h_{j,i}$'s are also fixed. For $j \in \alpha_1 \cup \alpha_2$ the values $b_{j,1}, b_{j,2}$ are fixed by $k_j$ as shown in the cases 1. and 2. above. For $j \in \alpha_0 \cup \alpha_{-1}$ the values $b_{j,2}$ are fixed by the choice of $k_j$ but $b_{j,1}$ is not, see cases 3. and 4.. For $j \in \alpha_{-2}$ both $b_{j,1}$ and $b_{j,2}$ are not fixed. Thus it follows that
\begin{align*}
& \sum_{j\in \alpha_0} \vec{c}_{j,b_{j,1} }^{\ \top}  + \sum_{j \in \alpha_{-1}} \vec{c}_{j,b_{j,1} }^{\ \top} + \sum_{j \in \alpha_{-2}} (\vec{c}_{j,b_{j,2} }^{\ \top} + \vec{c}_{j,b_{j,1} }^{\ \top}) \\  = & \sum_{j=1}^s \sum_{r=0}^{b_{j,2}-2} \vec{c}^{\ \top}_{j,r+1} h_{j,r} + \sum_{j \in \alpha_1 \cup \alpha_{2}} \vec{c}^{\ \top}_{j,b_{j,2}} =: \vec{c}^{\ \top},
\end{align*}
where the vector $\vec{c}^{\ \top}$ is fixed by $\bsk$, since the $h_{i,j}$ and $b_{j,1}, b_{j,2}$ are fixed by $\boldsymbol{k}$ for $j \in \alpha_1 \cup \alpha_2$. Since $\mu(l_j) = b_{j,1}$ for $1 \le j \le s$ we have $b_{1,1} + \cdots + b_{s,1} =z-z_1=: z_2$.

Since $h_j$ is fixed by $k_j$ for $1 \le j \le s$, it follows that for each given vector $(b_{j,i})_{1\le i \le 2, 1 \le j \le s}$, where $b_{j,1} > b_{j,2}$ and where $b_{1,1} + \cdots + b_{s,1} =z_2$, at most one such solution exists. Thus $|R_{i,i'}^{(1)}(\boldsymbol{k}, z_2)|$ is bounded by the number of possible choices of $(b_{j,i})_{1\le i \le 2, 1 \le j \le s}$, for which we prove a bound in the following.

The order $2$ and order $1$ digital $(t,m_i,s)$-net property and $\boldsymbol{l} \in \mathcal{D}_{i'}^\ast$ imply that
\begin{align*}
b_{1,1} + b_{1,2} + b_{2,1} + b_{2,2} + \cdots + b_{s,1} + b_{s,2} & > 2m_{i'}-t, \\
z_2 = b_{1,1} + b_{2,1} + \cdots + b_{s,1} & > m_{i'} - t.
\end{align*}
Thus we have
\begin{equation*}
b_{1,2} + \cdots + b_{s,2} \ge 2m_{i'} -t-z_2+1.
\end{equation*}
Let $b_{j,1} = \delta_j + b_{j,2}$, thus $\delta_j \ge 0$ (where $\delta_j=0$ if $l_j=0$). Then we have
\begin{eqnarray*}
z_2 & = & b_{1,1} + \cdots + b_{s,1} \\ & = &\delta_1+\cdots + \delta_s + b_{1,2} + \cdots + b_{s,2} \\ & \ge & \delta_1+\cdots + \delta_s + 2m_{i'} -t-z_2+1.
\end{eqnarray*}
and therefore
\begin{equation*}
\delta_1 + \cdots + \delta_s \le 2 z_2 - 2m_{i'} +t.
\end{equation*}
Thus, for given $b_{1,2}, b_{2,2}, \ldots, b_{s,2}$, the number of possible choice of $b_{1,1}, \ldots, b_{s,1}$ with $b_{1,1} + \cdots + b_{s,1} = z_2$ is bounded by the number of possible choices of $\delta_1,\delta_2, \ldots, \delta_s$, which itself is bounded from above by
\begin{equation*}
\sum_{r=0}^{2z_2-2m_{i'} +t} {r+s-1 \choose s-1} = {2z_2-2m_{i'} +t+ s \choose s}.
\end{equation*}

Now consider the number of possible choices of $(b_{j,2})_{1\le j \le s}$. If $j \in \bigcup_{\tau=-1}^2 \alpha_\tau$, then $b_{j,2}$ is fixed since $k_j$ is fixed and if $j \in \alpha_{-2}$, then $b_{j,1} > b_{j,2} > b_{j,3} = a_{j,1}$. Note that $b_{j,3}$ is fixed since $k_j$ is fixed for all $1 \le j \le s$ . By the order $3$, order $2$ and order $1$ digital net property and $\boldsymbol{l} \in \mathcal{D}_{i'}^\ast$ we have
\begin{eqnarray*}
b_{1,1} + b_{1,2} + b_{1,3} + \cdots + b_{s,1} + b_{s,2} + b_{s,3} & > & 3m_{i'}-t, \\
b_{1,1} + b_{1,2} + \cdots + b_{s,1} + b_{s,2} & > & 2m_{i'} -t, \\
z_2 = b_{1,1} + \cdots + b_{s,1} & > & m_{i'} -t.
\end{eqnarray*}
Let $z_2' = b_{1,2} + b_{2,2} + \cdots + b_{s,2} < z_2$. Then
\begin{equation*}
b_{1,3} + \cdots + b_{s,3} > 3m_{i'} -t - z_2 - z'_2 > 3m_{i'} -t-2z_2.
\end{equation*}
Let $b_{j,2} = \delta'_j + b_{j,3}$, then $\delta'_j \ge 0$. Then we have
\begin{eqnarray*}
z_2 & > & b_{1,2} + \cdots + b_{s,2} \\ & = & \delta'_1 + \cdots + \delta'_s + b_{1,3} + \cdots + b_{s,3} \\ & \ge & \delta'_1 + \cdots + \delta'_s + 3m_{i'} -t-2z_2+1
\end{eqnarray*}
and therefore
\begin{equation*}
\delta'_1 + \cdots + \delta'_s \le 3z_2 - 3m_{i'} +t-1.
\end{equation*}
Since the $b_{1,3}, b_{2,3}, \ldots, b_{s,3}$ are fixed, the number of admissible $b_{1,2} ,\ldots, b_{s,2}$ is bounded from above by the number of possible choices of $\delta'_1,\ldots, \delta'_s$, which in turn is bounded by
\begin{equation*}
\sum_{r=0}^{3z_2 - 3m_{i'} +t} {r+s-1 \choose s-1} = {3z_2 - 3m_{i'} +t + s \choose s}.
\end{equation*}

Since the number of possible choices of $(b_{j,i})_{1 \le i \le 2, 1 \le j \le s}$ is bounded by the product of the number of possible choices of $b_{1,1}, b_{2,1}, \ldots, b_{s,1}$ and the number of possible choices of $b_{1,2}, b_{2,2}, \ldots, b_{s,2}$, we deduce
\begin{eqnarray*}
|R_{i,i'}^{(1)}(\boldsymbol{k},z_2)| \le   {2z_2-2m_{i'} +t+ s \choose s}{3z_2 - 3m_{i'} +t + s \choose s}.
\end{eqnarray*}
Thus the statement of the lemma follows.
\end{proof}

Before we combine Lemmas~\ref{lem_bound_D} and \ref{lem_bound_R} to obtain a bound on $|\mathcal{J}_{i,i'}(z)|$ we show that for `small' $z$ the set $\mathcal{J}_{i,i'}(z)$ is empty in the next lemma. In the proof of this lemma we need to assume that $\alpha \ge 5$. %This is also the only place where $\alpha \ge 5$ is required, the proof the remaining statements requires at most $\alpha \ge 3$.

\begin{lemma}\label{lem_J_empty}
Let $\alpha \ge 5$. Then we have $\mathcal{J}_{i,i'}(z) = \emptyset$ if $z < \frac{1}{4} \max\{5m_i + 3m_{i'}, 3m_i + 5m_{i'}\} - t + \frac{3}{4}$.
\end{lemma}

\begin{proof}
We use the notation from the proof of Lemma~\ref{lem_bound_R}.

Assume that $(\bsk, \bsl) \in \mathcal{J}_{i,i'}(z)$. Consider again the five cases from the proof of Lemma~\ref{lem_bound_R}. The following holds: 
\begin{enumerate}
\item For $j \in \alpha_2$ we have $a_{j,i+2} = b_{j,i}$ for $i=1,\ldots,\widetilde{w}_j$ and $\widetilde{w}_j=\widetilde{v}_j-2$; 
\item For $j \in \alpha_1$ we have $a_{j,i+1} = b_{j,i} = 0$ for $i=1,\ldots,\widetilde{w}_j$ and $\widetilde{w}_j=\widetilde{v}_j-1$; 
\item For $j \in \alpha_0$ we have $a_{j,i+1} = b_{j,i+1}$ for $i=1,\ldots,\widetilde{w}_j$ and $\widetilde{w}_j=\widetilde{v}_j$; 
\item For $j \in \alpha_{-1}$ we have $a_{j,i} = b_{j,i+1} = 0$ for $i=1,\ldots,\widetilde{v}_j$ and $\widetilde{v}_j=\widetilde{w}_j-1$; 
\item For $j \in \alpha_{-2}$ we have $a_{j,i} = b_{j,i+2}$ for $i=1,\ldots,\widetilde{v}_j$ and $\widetilde{v}_j=\widetilde{w}_j-2$.
\end{enumerate}
Since $a_{j,i} > a_{j,i+1}$ we therefore have $b_{j,3} \ge a_{j,5}$ for $1\le j \le s$. By the order $5$ digital $(t,m,s)$-net property we have
\begin{align*}
& a_{1,1} + a_{1,2} + a_{1,3} + a_{1,4} + a_{1,5} + \cdots + a_{s,1} + a_{s,2} + a_{s,3} + a_{s,4} + a_{s,5}   > 5 m_i -t, \\
& z_1 = a_{1,1} + \cdots + a_{s,1} \ge a_{1,2} + \cdots + a_{s,2} \ge a_{1,3} + \cdots + a_{s,3} \ge a_{1,4} + \cdots + a_{s,4}.
\end{align*}
Thus and since $b_{j,i} > b_{j,i+1}$ we obtain
\begin{equation*}
z-z_1 = z_2 = b_{1,1} + \cdots + b_{s,1} \ge b_{1,3} + \cdots + b_{s,3} \ge a_{1,5}  + \cdots + a_{s,5} \ge 5m_i-t - 4z_1.
\end{equation*}
From the proof of Lemma~\ref{lem_bound_factors} we have that $z-z_1 \ge m_{i'}-t+1$, therefore
\begin{equation*}
z \ge 5m_i-t - 3z_1 \ge 5m_i - t + 3(m_{i'}-t + 1-z),
\end{equation*}
which implies
\begin{equation*}
z \ge \frac{5 m_i + 3 m_{i'}}{4} - t + \frac{3}{4}.
\end{equation*}
Analogously we have
\begin{equation*}
z \ge \frac{5 m_i + 3 m_{i'}}{4} - t + \frac{3}{4}.
\end{equation*}
Thus we have $\mathcal{J}_{i,i'}(z) = \emptyset$ if $z < \frac{1}{4} \max\{5m_i + 3m_{i'}, 3m_i + 5m_{i'}\} - t + \frac{3}{4}$.
\end{proof}

In the following we also obtain a bound on $|\mathcal{J}_{i,i'}(z)|$ for $z \ge m_i + m_{i'}-2t+2$. In Lemma~\ref{lem_J_empty} we considered $z < \frac{1}{4} \max\{5m_i + 3m_{i'}, 3m_i + 5m_{i'}\} - t + \frac{3}{4}$. At the beginning of this section we showed that $t \ge 3$. Since $\frac{1}{4} \max\{5m_i + 3m_{i'}, 3m_i + 5m_{i'}\} - t + \frac{3}{4} \ge m_i + m_{i'}-2t+2$ for $t \ge 3$, Lemma~\ref{lem_J_empty} and Lemma~\ref{lem_J_bound} yield a bound on $|\mathcal{J}_{i,i'}(z)|$ for all $ z \ge 0$.

\begin{lemma}\label{lem_J_bound}
For all $\kappa \ge 0$ we have
\begin{eqnarray*}
\lefteqn{ |\mathcal{J}_{i,i'}(m_i + m_{i'} -2t+2 + \kappa)| } \\ & \le & 2^{\kappa/2+2} \sum_{z'_1=0}^{\lceil \kappa/2 \rceil}  {z'_1 + m_i -t +s \choose s-1}  \\ & & \hspace{2cm} \times {2 (\kappa-z_1')+s+2-t \choose s}  {3 (\kappa-z_1')+s+3-2t \choose s}.
\end{eqnarray*}
\end{lemma}

\begin{proof}
Combining Lemmas~\ref{lem5neu}, \ref{lem_bound_D} and \ref{lem_bound_R} we obtain
\begin{eqnarray*}
\lefteqn{|\mathcal{J}_{i,i'}(z)|}\\
 & \le &   \sum_{z_1=m_i-t+1}^{z-m_{i'} +t-1} \min \\
&& \left\{ 2^{z_1 - m_i + t - 1} {z_1+s-1 \choose s-1} {2z-2z_1-2m_{i'} +t + s \choose s}  {3z-3z_1 - 3m_{i'} +t+ s \choose s},\right. \\
& & \left. 2^{z-z_1-m_{i'}+t-1} {z-z_1+s-1 \choose s-1}{2z_1-2m_{i} +t + s \choose s}  {3z_1 - 3m_{i} +t+ s \choose s}\right\}.
\end{eqnarray*}

To simplify this bound further we first use the change of variable $z=m_i + m_{i'} - 2t + 2 + \kappa$ for $\kappa \ge 0$. Then we have
\begin{eqnarray*}
\lefteqn{|\mathcal{J}_{i,i'}(m_i + m_{i'} -2t+2+\kappa)|}\\
 & \le &   \sum_{z_1'=0}^{\kappa} \min \\
&& \left\{ 2^{z'_1}   {z'_1 + m_i -t +s \choose s-1} {2(\kappa-z'_1) + s+2-t \choose s}  {3 (\kappa-z_1')+s+3-2t \choose s},\right. \\
& & \left. 2^{\kappa-z'_1} {\kappa-z'_1+m_{i'}-t+s \choose s-1}{2z'_1+s +2-t  \choose s}  {3z'_1 + s+3-2t \choose s}\right\}.
\end{eqnarray*}
Let
\begin{eqnarray*}
B(z'_1,z'_2) = 2^{z'_1}   {z'_1 + m_i -t +s \choose s-1} {2z'_2 + s+2-t \choose s}  {3 z_2'+s+3-2t \choose s}.
\end{eqnarray*}
Then we obtain
\begin{eqnarray}\label{bound_sym}
\lefteqn{|\mathcal{J}_{i,i'}(m_i + m_{i'} -2t+2+\kappa)|}\\
& \le & \sum_{z'_1=0}^{\kappa} \min\{B(z'_1,\kappa-z'_1), B(\kappa-z'_1, z'_1)\} \nonumber \\   & \le &  2 \sum_{z'_1=0}^{\lceil \kappa / 2 \rceil} \min\{B(z'_1,\kappa-z'_1), B(\kappa-z'_1, z'_1)\} \nonumber \\  & \le & 2 \sum_{z'_1=0}^{\lceil \kappa/2 \rceil} B(z'_1, \kappa-z'_1) \nonumber \\  & \le & 2^{\kappa/2+2} \sum_{z'_1=0}^{\lceil \kappa/2 \rceil}  {z'_1 + m_i -t +s \choose s-1} \nonumber \\ & & \hspace{2cm} \times {2 (\kappa-z_1')+s+2-t \choose s}  {3 (\kappa-z_1')+s+3-2t \choose s}. \nonumber
\end{eqnarray}
\end{proof}

The following lemma now implies Theorem~\ref{thm0}. Since the proof makes use of Lemma~\ref{lem_J_empty} we need to assume that $\alpha \ge 5$.

\begin{lemma}\label{lem_proof_final}
Let $\alpha \ge 5$. Let $N = 2^{m_1} + 2^{m_2} + \cdots + 2^{m_r} \ge 2$ with $m_1 > m_2 > \ldots > m_r \ge 0$. Then we have
\begin{equation*}
\LL_{2,N}^2(\cS_s) \ll_s \frac{(\log N)^{s-1}}{N} r.
\end{equation*}
\end{lemma}

\begin{proof}
Assume that $i \le i'$. Note that for $\kappa < \lfloor (m_i-m_{i'})/4 +t-5/4 \rfloor$ we have $\mathcal{J}_{i,i'}(m_i+m_{i'}-2t+2+\kappa) = \emptyset$ by Lemma~\ref{lem_J_empty}. Now we use Lemma~\ref{lem_J_bound} to obtain for the innermost sum in Lemma~\ref{lem_proof_1} that % \eqref{def_sum_bound2}
\begin{eqnarray*}
\lefteqn{\sum_{z=m_i+m_{i'} - 2 t+2}^\infty \frac{|\mathcal{J}_{i,i'}(z)|}{2^z} }\\ & \ll_s &  \frac{1}{2^{m_i+m_{i'} }}  \sum_{\kappa= \lfloor (m_i-m_{i'})/4 + t -5/4 \rfloor}^\infty \frac{|\mathcal{J}_{i,i'}(m_i + m_{i'} -2t + 2 + \kappa)|}{2^{\kappa}} \\ & \ll_s & \frac{1}{2^{m_i+m_{i'} }}  \sum_{\kappa=  \lfloor (m_i-m_{i'})/4 + t -5/4 \rfloor}^\infty \frac{2^2}{2^{\kappa/2}} \sum_{z'_1=0}^{\lceil \kappa/2\rceil} {z'_1+m_i-t+s \choose s-1} \\ && \times { 2(\kappa-z'_1)+s+2-t \choose s} {3(\kappa-z'_1) + s + 3-2t \choose s}.
\end{eqnarray*}
Since $t$ depends only on the dimension $s$ but not on $m_i, m_{i'}$, we can simplify the above expression to obtain
\begin{eqnarray*}
\lefteqn{\sum_{z=m_i+m_{i'} - 2 t+2}^\infty \frac{|\mathcal{J}_{i,i'}(z)|}{2^z} } \\ & \ll_s & \frac{1}{2^{m_i+m_{i'} }}  \sum_{\kappa=\lfloor (m_i-m_{i'})/4 \rfloor }^\infty \frac{1}{2^{\kappa/2}} \sum_{z'_1=0}^{\lceil \kappa/2\rceil} {z'_1+m_i \choose s-1}  { 2(\kappa-z'_1) \choose s} {3(\kappa-z'_1) \choose s}.
\end{eqnarray*}
We estimate the binomial coefficients using $0 \le z'_1 \le \kappa$ to obtain
$${z'_1+m_i \choose s-1} \ll_s (m_i+1)^{s-1} (z'_1+1)^{s-1} \ll_s (\log N)^{s-1} (\kappa+1)^{s-1}$$ 
and $${2(\kappa-z'_1) \choose s} {3(\kappa-z'_1) \choose s} \ll_s (\kappa+1)^{2s}.$$
Thus we have
\begin{eqnarray*} 
\sum_{z=m_i+m_{i'} - 2 t+2}^\infty \frac{|\mathcal{J}_{i,i'}(z)|}{2^z}   & \ll_s & \frac{(\log N)^{s-1}}{2^{m_i+m_{i'} }}  \sum_{\kappa= \lfloor (m_i-m_{i'})/4 \rfloor }^\infty \frac{(\kappa+1)^{3s} }{2^{\kappa/2}}.
\end{eqnarray*}
Inserting this bound into Lemma~\ref{lem_proof_1} we obtain
\begin{eqnarray*}
\LL_{2,N}^2(\cS_s) & \ll_s  & \frac{(\log N)^{s-1}}{N^2} \sum_{1 \le i \le i' \le r}  \sum_{\kappa= \lfloor (m_i-m_{i'})/4 \rfloor }^\infty \frac{(\kappa+1)^{3s} }{2^{\kappa/2}}.
\end{eqnarray*}
Using the fact that for $i \le i'$ we have $m_i \ge m_{i'}$ we obtain for any fixed $1 \le i \le r$ that
\begin{eqnarray*}
\sum_{i'=i}^r \sum_{\kappa= \lfloor (m_i-m_{i'})/4 \rfloor }^\infty \frac{(\kappa+1)^{3s} }{2^{\kappa/2}} & \ll_s & \sum_{i'=i}^r \sum_{\kappa= \lfloor (m_i-m_{i'} )/ 4 \rfloor }^\infty \frac{1}{2^{\kappa/4}} \\ & \ll & \sum_{i'=i}^r \frac{1}{2^{(m_i-m_{i'})/16}} \\ & \le & \sum_{q = 0}^\infty \frac{1}{2^{q/16}} \\ & \ll & 1.
\end{eqnarray*}
Thus we obtain $$\sum_{1 \le i \le i' \le r}  \sum_{\kappa= \lfloor (m_i-m_{i'})/4 \rfloor }^\infty \frac{(\kappa+1)^{3s} }{2^{\kappa/2}} \ll_s r$$ and therefore
\begin{eqnarray*}
\LL_{2,N}^2(\cS_s) & \ll_s  & \frac{(\log N)^{s-1}}{N^2} r,
\end{eqnarray*}
where $r=S(N)$ denotes the number of nonzero digits in the binary expansion of $N$.% Thus the result follows by taking the square root.
\end{proof}

\section{The proof of Corollary~\ref{thm0}}

To prove Corollary~\ref{thm0} we first prove a bound on the $\LL_2$ discrepancy of order $3$ digital nets.

\subsection{A bound on the $\LL_2$ discrepancy of order $3$ digital nets}

\begin{theorem}\label{thm_net}
Let $s,m \in \NN$. For every (digitally shifted) order $3$ digital $(t,m,s)$-net $\cP_{2^m,s}$ over $\mathbb{F}_2$ we have
\begin{equation*}
\LL_{2,2^m}(\cP_{2^m,s}) \ll_s \frac{m^{(s-1)/2}}{2^{m-t}}.
\end{equation*}
\end{theorem}

\begin{proof}
The proof of Theorem~\ref{thm_net} can be obtained by specializing the proof of Theorem~\ref{thm0} to the case where $r=1$. In the following we describe the necessary changes in the proof of Theorem~\ref{thm0} to obtain the result. The reason for requiring only $\alpha =3$ instead of $\alpha \ge 5$ is that we do not make use of Lemma~\ref{lem_J_empty} in this proof.

Let $C_1,\ldots, C_s \in \mathbb{F}_2^{3 m \times m}$ be the generating matrices of $\cP_{2^ms}$ and recall the definition
\begin{equation*}
\mathcal{D} = \{\boldsymbol{k} \in \mathbb{N}_0^s: C_1^\top \vec{k}_1 + \cdots + C_s^\top \vec{k}_s \equiv 0 \pmod{2}\}
\end{equation*}
and $\mathcal{D}^\ast = \mathcal{D} \setminus \{\boldsymbol{0}\}$. We can use the same argument as in the proof of Theorem~\ref{thm0} where $r=1$. Let $\Jcal=\Jcal_{i,i'}$ and $\mathcal{J}(z)=\Jcal_{i,i'}(z)$ from the proof of Theorem~\ref{thm1} with $i=i'$ and $m_i=m$, i.e., by \eqref{def_Jii} and \eqref{def_Jii_z} we have
\begin{equation*}
\Jcal = \{(\bsk,\bsl) \in \Dcal^\ast \times \Dcal^\ast \ : \ r(\bsk,\bsl) \neq 0\}
\end{equation*}
and
\begin{equation*}
\mathcal{J}(z) = \{(\boldsymbol{k},\boldsymbol{l}) \in \mathcal{J} \ : \ \mu(\boldsymbol{k}) + \mu(\boldsymbol{l}) = z\}.
\end{equation*}

By the same arguments as in the proof of Theorem~\ref{thm0}, see \eqref{def_sum_bound}, we have
\begin{equation*}
\LL_{2,2^m}^2(\cP_{2^m,s}) = \left| \sum_{\bsk,\bsl \in \Dcal^\ast} r(\bsk,\bsl) \wal_{\bsk}(\bssigma) \wal_{\bsl}(\bssigma) \right| \le \sum_{\bsk,\bsl \in \Dcal^\ast} |r(\bsk,\bsl)| = \sum_{(\bsk,\bsl) \in \Jcal} |r(\bsk,\bsl)|
\end{equation*}
and for $(\bsk,\bsl) \in \Jcal$ we have $$|r(\bsk,\bsl)| \le \frac{1}{3^s 2^{\mu(\bsk) + \mu(\bsl)}}.$$  Thus we have (cf. \eqref{def_sum_bound2}) $$\LL_{2,2^m}^2(\cP) \ll_s  \sum_{(\bsk,\bsl) \in \Jcal} \frac{1}{2^{\mu(\bsk) + \mu(\bsl)}}.$$

It follows from the (order $1$) digital $(t,m,s)$-net property and  $\bsk \in \Dcal^\ast$ that $\mu(\boldsymbol{k}) > m-t$ and from $\bsl \in \Dcal^\ast$ it follows also that $\mu(\boldsymbol{l}) > m-t$ and hence $\mu(\bsk) + \mu(\bsl) \ge 2(m-t+1)$. Therefore (cf. Lemma~\ref{lem_proof_1}) 
\begin{equation}\label{bdL2net1}
\LL_{2,2^m}^2(\cP) \ll_s  \sum_{z=2(m-t+1)}^\infty \frac{|\Jcal(z)|}{2^z}.
\end{equation}

From Lemma~\ref{lem_J_bound}, we obtain
for $z=2m-2t + 2 + \kappa$ for $\kappa \ge 0$ that

\begin{eqnarray*}
|\mathcal{J}(2m-2t+2+\kappa)| & \le &  2^{\kappa/2+2} \sum_{z'_1=0}^{\lceil \kappa/2 \rceil}  {z'_1 + m-t +s \choose s-1} \\ & & \times {2(\kappa-z'_1) +s+2-t \choose s}  {3(\kappa-z'_1)+s+3-2t \choose s}.
\end{eqnarray*}

Inserting this result into \eqref{bdL2net1} we obtain
\begin{eqnarray*}
\LL_{2,2^m}^2(\cP_{2^m,s})  & \ll_s  & \frac{1}{2^{2m-2t+2}}  \sum_{\kappa=0}^\infty \frac{|\mathcal{J}(2m-2t+2+\kappa)|}{2^{\kappa}} \\ & \le & \frac{1}{2^{2m-2t}}  \sum_{\kappa=0}^\infty \frac{1}{2^{\kappa/2}} \sum_{z'_1=0}^{\lceil \kappa/2 \rceil}  {z'_1 + m-t +s \choose s-1} \\ & & \hspace{1cm}\times {2(\kappa-z'_1) +s+2-t \choose s}  {3(\kappa-z'_1)+s+3-2t \choose s} \\ & \le  &  \frac{1}{2^{2m-2t}}  \sum_{\kappa=0}^\infty \frac{\kappa/2+1}{2^{\kappa/2}}  \frac{(\kappa + 1 + m-t +s)^{s-1}}{(s-1)!} \\ & & \hspace{1cm} \times\frac{(2\kappa +s+2-t)^s}{s!}  \frac{(3\kappa+s+3-2t)^s}{s!}.
\end{eqnarray*}
Since the sum over $\kappa$ is now from $0$ to $\infty$, we do not need to use Lemma~\ref{lem_J_empty}. Thus also the assumption that $\alpha \ge 5$ is not needed and $\alpha = 3$ is sufficient.

Using the fact that $t$ depends only on the dimension $s$, we therefore obtain
\begin{eqnarray*}
\LL_{2,2^m}^2(\cP_{2^m,s}) \ll_s \frac{m^{s-1}}{2^{2m}}  \sum_{\kappa=0}^\infty \frac{\kappa^{3s}}{2^{\kappa/2}}   \ll_s  \frac{m^{s-1}}{2^{2m}}.
\end{eqnarray*}
Thus the result follows by taking the square root.
\end{proof}

\subsection{The proof of Corollary~\ref{thm0}}

The proof of Corollary~\ref{thm0} uses Theorem~\ref{thm_net} and an idea of \cite{CS02}. %The explicit construction of the point sets with optimal $\LL_2$ discrepancy with respect to the order of magnitude in $N \ge 2$ is outlined in the proof below.

\begin{proof}
For an integer $N \ge 2$ we choose $m \in \NN$ such that $2^{m-1} < N \le 2^{m}$. Let $\cP_{2^m,s}$ be an order $3$ digital $(t,m,s)$-net over $\mathbb{F}_2$ with the property that the first component of $\cP_{2^m,s}$ is a $(0,m,1)$-net over $\mathbb{F}_2$. Note that such nets exist for every $m$ and can be obtained in the following way: Take the digital sequence introduced in Section~\ref{sec_construction} in dimension $3s-1$. Concatenate to the $n$th element the component $n 2^{-m}$ for $n=0,1,\ldots,2^m-1$, so that the new points are of the form $(n2^{-m}, y_{n,1}, y_{n,2},\ldots, y_{n,3s-1})$, where $(y_{n,1},\ldots, y_{n,3s-1})$ is the $n$th point of the sequence. Then the set consisting of the points $(n2^{-m},y_{n,1}, y_{n,2})$ for $0 \le n < 2^m$ is a digital $(0,m,3)$-net. Apply the digit interlacing composition to the point set $$\{(n2^{-m}, y_{n,1}, y_{n,2},\ldots, y_{n,3s-1})\, : \ n=0,1,\ldots,2^m -1\}.$$

We can now use \cite[Proposition~1]{D10}, which states the following: Let $C_1,\ldots, C_{\alpha s}$ be the generating matrices of a digital $(t,m,s)$-net. Let $C^{(\alpha)}_1, \ldots, C^{(\alpha)}_s$ be the matrices obtained by applying the interlacing construction to $C_1,\ldots, C_{\alpha s}$. Then $C^{(\alpha)}_1,\ldots, C^{(\alpha)}_s$ are generating matrices of an order 3 digital $(t,m,s)$-net. In particular, it follows that the first component of the order $3$ digital net obtained this way is a digital $(0,m,1)$-net. 

We now proceed as in \cite{CS02}. According to Theorem~\ref{thm_net} we have
\begin{equation}\label{bdl2discnet2}
\mathcal{L}_{2,2^m}(\cP_{2^m,s}) \ll_s \frac{m^{(s-1)/2}}{2^{m}}.
\end{equation}
As shown above, the first component of $\cP_{2^m,s}$ is a digital
$(0,m,1)$-net over $\FF_2$. Hence the subset
$$\widetilde{\cP}_{N,s}:=\cP_{2^m,s} \cap
\left(\left[0,\frac{N}{2^m}\right) \times [0,1)^{s-1}\right)$$
contains exactly $N$ points. We define the point set
$$\cP_{N,s}:=\left\{\left(\frac{2^m}{N} x_1,x_2,\ldots,x_s\right)\, :
\, (x_1,x_2,\ldots,x_s) \in \widetilde{\cP}_{N,s}\right\}.$$ Then we
have (where $\bsy=(y_1,\ldots,y_s)$)
\begin{eqnarray*}
\lefteqn{(N \LL_{2,N}(\cP_{N,s}))^2  =  \int_{[0,1]^s} \left|A\left([\bszero,\bsy),N,\cP_{N,s})\right)-N \lambda_s([\bszero,\bsy)) \right|^2 \rd \bsy}\\
& = &  \int_0^1 \cdots \int_0^1 \left|A\left(\left[0,y_1 N2^{-m}\right) \times \prod_{i=2}^s[0,y_i),N,\widetilde{\cP}_{N,s}\right)\right.\\
&& \hspace{3cm}\mbox{}-2^m \left. \frac{N}{2^m} y_1\cdots y_s \right|^2 \rd y_1\cdots \rd y_s\\
& = & \frac{2^m}{N} \int_0^{N/2^m}\int_0^1 \cdots \int_0^1\left|A([\bszero,\bsy),N,\widetilde{\cP}_{N,s})-2^m \lambda_s([\bszero,\bsy)) \right|^2 \rd \bsy\\
& = & \frac{2^m}{N} \int_0^{N/2^m}\int_0^1 \cdots \int_0^1\left|A([\bszero,\bsy),2^m,\cP_{2^m,s})-2^m \lambda_s([\bszero,\bsy)) \right|^2 \rd \bsy\\
& \le & \frac{2^m}{N} \left(2^m \LL_{2,2^m}(\cP_{2^m,s})\right)^2.
\end{eqnarray*}
With \eqref{bdl2discnet2} we obtain
\begin{eqnarray*}
(N \LL_{2,N}(\cP_{N,s}))^2  \ll_s \frac{2^m}{N} m^{s-1} \ll_s  (\log N)^{s-1}.
\end{eqnarray*}
Taking the square root and dividing by $N$ we finally obtain $$\LL_{2,N}(\cP_{N,s}) \ll_s  \frac{\left(\log N\right)^{(s-1)/2}}{N}.$$
\end{proof}

\section{Acknowledgements}

H. Niederreiter also independently suggested recently that higher order nets may achieve the optimal rate of convergence of the $\mathcal{L}_2$ discrepancy.

%\vspace{1cm}
\noindent{\bf Author's Addresses:}

\noindent Josef Dick, School of Mathematics and Statistics, The University of New South Wales, Sydney, NSW 2052, Australia.  Email: josef.dick@unsw.edu.au \\

\noindent Friedrich Pillichshammer, Institut f\"{u}r Analysis, Universit\"{a}t Linz, Altenbergerstra{\ss}e 69, A-4040 Linz, Austria. Email: friedrich.pillichshammer@jku.at

\end{document}